\documentclass[12pt]{article}

\usepackage[margin=1in]{geometry}  
\usepackage{amsmath,amsfonts,amsthm,amssymb}	
\usepackage{tikz}							
\usepackage[shortlabels]{enumitem}		
\usepackage{hyperref}
\usetikzlibrary{calc}

\theoremstyle{plain}
\newtheorem{thm}{Theorem}[section]
\newtheorem{lem}[thm]{Lemma}
\newtheorem{prop}[thm]{Proposition}
\newtheorem{cor}[thm]{Corollary}
\newtheorem{ques}[thm]{Question}

\theoremstyle{definition}
\newtheorem{df}[thm]{Definition}
\newtheorem{exmp}[thm]{Example}
\newtheorem{rmk}[thm]{Remark}

\newcommand{\comment}[1]{}

\newcommand{\abs}[1]{\left|#1\right|}
\newcommand{\al}{\alpha}
\newcommand{\bbN}{\mathbb{N}}
\newcommand{\bbR}{\mathbb{R}}

\newcommand{\calC}{\mathcal{C}}

\newcommand{\calU}{\mathcal{U}}
\newcommand{\conv}{\textnormal{conv}}
\newcommand{\cont}{\gamma}
\newcommand{\D}{\Delta}
\newcommand{\emp}{\emptyset}
\newcommand{\geo}[1]{\|#1\|}
\renewcommand{\hat}[1]{\widehat{#1}}

\newcommand{\ind}[1]{\textnormal{ind}_{\mathbb{Z}_2}(#1)}
\renewcommand{\int}[1]{\left[#1\right]}
\newcommand{\m}[1]{\mathsf m_{#1}}
\newcommand{\Par}[1]{\left(#1\right)}
\newcommand{\pathlet}{\alpha}
\renewcommand{\phi}{\varphi}
\newcommand{\Rt}{\mathbb{R}^2}
\newcommand{\Rd}{\mathbb{R}^d}

\newcommand{\simp}[1]{\mathsf{#1}}
\newcommand{\scdp}[1]{\geo{{#1}^2_\D}}
\newcommand{\sd}{\textnormal{sd}}
\newcommand{\set}[1]{\left\{#1\right\} }
\newcommand{\su}{\subseteq}
\newcommand{\supp}{\textnormal{supp}}
\newcommand{\topdp}[1]{{#1}^2_\D}
\newcommand{\Zt}{\mathbb{Z}_2}

\comment{
adjacencies
advisor
asteroidal
asteroidality
barycentric
Boland
Borsuk-Ulam
collapses
collapsibility
combinatorial
contractibility
Eckhoff
embeddability
equivariant
equivariantly
Helly
Helly's
Kalai
Kampen
homeomorphic
homeomorphism
homological
homology
homotopy
iff
injective
isomorphism
Lekkerkerker
Leray
Matousek = Matou\v sek
Pavle Blagojevi\'c
piecewise
planarity
precomposed
representable
representability
Skopenkov
Skopenkov's
simplices
simplicial
subcollections
subcomplex
subgraph
Tancer
Tancer's
Wegner
Wegner's
}

\numberwithin{equation}{section}

\newcounter{figure_num}

\begin{document}

\title{$d$-representability as an embedding problem}
\author{Moshe J. White \thanks{Research supported in part by the Israel Science Foundation (Grant No. 2669/21) and in part by the ERC Advanced Grant (Grant No. 834735)} \\
Institute of Mathematics,\\
Hebrew University, Jerusalem, Israel}
\date{}
\maketitle


\begin{abstract}
An abstract simplicial complex is said to be $d$-representable if it records the intersection pattern of a collection of convex sets in $\mathbb{R}^d$. In this paper, we show that $d$-representability of a simplicial complex is equivalent to the existence of a map with certain properties, from a closely related simplicial complex into $\mathbb{R}^d$. This equivalence suggests a framework for proving (and disproving) $d$-representability of simplicial complexes using topological methods such as applications of the Borsuk--Ulam theorem, which we begin to explore.
\end{abstract}

\section{Introduction} 
The \emph{nerve} of any finite collection of sets $C_1,\ldots,C_n$, is a simplicial complex which records all subcollections with non-empty intersection: the faces are precisely those $I\su\set{1,\ldots,n}$ for which $\bigcap_{i\in I}C_i\neq\emp$. 
A simplicial complex $\simp K$ is said to be \emph{$d$-representable} if it is the nerve of a collection of convex sets in $\Rd$.


The general problem of classifying $d$-representable complexes is wide open, but many partial results are known. Perhaps the earliest result on $d$-representability is Helly's classical theorem from 1923. There are additional topological restrictions: a $d$-representable complex cannot have non-trivial homology in any dimension $m\geq d$, and the same holds for any induced subcomplex. A complex with these homological properties is said to be \emph{$d$-Leray}. Another important property, combinatorial in nature, was introduced by Wegner \cite{We} in 1975. In an \emph{elementary $d$-collapsing} of a simplicial complex $\simp K$, we take a face $F$ of $\simp K$ which has dimension $<d$ and is contained in a unique maximal face of $\simp K$, and delete from $\simp K$ the face $F$ and every face containing it. If there is a sequence of elementary $d$-collapses from $\simp K$ to $\emp$, we say that $\simp K$ is \emph{$d$-collapsible}. Wegner proved that every $d$-representable complex is $d$-collapsible, and every $d$-collapsible complex is $d$-Leray. This result was extended by Eckhoff \cite{Eck}, who introduced \emph{strong $d$-collapsibility}, a property implied by $d$-representability and implying $d$-collapsibility.\\
For further reading on representability and related properties, the reader is referred to a survey by Tancer \cite{TaS}.

Therefore, if we discover that a certain complex $\simp K$ is, say, not $d$-collapsible, we may deduce that $\simp K$ is not $d$-representable. Another approach for disproving $d$-representability appeared in Wegner's PhD thesis from 1967, and was later extended by Tancer \cite{TaRep}. In the simplest case, they prove that a certain $\simp K$ (the barycentric subdivision of $K_5$, the complete graph on $5$ vertices) is not $2$-representable. This is done by showing that if $\simp K$ is the nerve of convex sets in $\Rt$, the vertices and edges of $K_5$ can be carefully mapped within the convex sets, in a manner implying that $K_5$ is a planar graph.


Our setting can be seen as a generalization of this result. For every $\simp K$ we define
the dual simplicial complex $\simp K'$ as the nerve of the facets of $\simp K$. Explicitly:
\begin{df} \label{dual_complex}
Let $\m{\simp K}$ denote the set of inclusion-maximal faces of a simplicial complex $\simp K$. The dual of $\simp K$ is defined as:
$$\simp K' = \set{\al\su \m{\simp K}\ :\ \bigcap\al\neq\emp}.$$
\end{df}
In Theorem \ref{faithful_lemma}, we prove that $\simp K$ is $d$-representable if and only if there exists a map from $\geo{\simp K'}$ (the geometric realization of $\simp K'$) to $\Rd$, satisfying certain properties. As in Tancer's proof, this is done by mapping the faces of $\simp K'$ within corresponding intersections of those convex sets in $\Rd$ which represent $\simp K$, whenever such sets exist. We also mention that a precursor of the idea to construct a map $\geo{\simp K'}\to\Rd$ by placing points among certain intersections of a collection of sets in $\Rd$, appeared in an earlier work of Matou\v sek \cite{MaH}.

This method can be used to obtain either representability or non-representability results, by either proving (as in Theorem \ref{codim_rep}) or disproving (as in Theorem \ref{non_rep}) the existence of the suitable map. These theorems may by applied regardless of the method by which we prove or disprove the existence of the required map $\geo{\simp K'}\to\Rd$. Proceeding with this goal, we will mostly focus on a topological method described in a book by Matou\v sek \cite{Ma}, on the application of the Borsuk--Ulam theorem. This leads to the introduction of a new class of simplicial complexes, which we name \emph{$d$-Matou\v sek} complexes:
\begin{df} \label{configuration_space}
Let $\simp K$ be a simplicial complex. We define $\hat{\simp K}$ as the topological space:
$$\hat{\simp K} = \set{(x,y)\in\geo{\simp K'}^2\ :\ \Par{\bigcap\ \supp(x)}\cup\Par{\bigcap\ \supp(y)}\notin \simp K}.$$
\end{df}
\begin{df} \label{d-Mat}
Let $\simp K$ be a simplicial complex, and let $d\in\bbN$. We say that $\simp K$ is $d$-Matou\v sek, if
there exists a continuous map $f:\hat{\simp K}\to S^{d-1}$ satisfying ${f(x,y)=-f(y,x)}$ for every ${(x,y)\in\hat{\simp K}}$.
\end{df}

Using the terminology of $\Zt$-spaces, the configuration space $\hat{\simp K}$ has an involution given by swapping coordinates, and $\simp K$ is $d$-Matou\v sek if the $\Zt$-index of $\hat{\simp K}$ is at most $d-1$. 

In broad terms, the configuration space $\hat{\simp K}$ consists of pairs of points which cannot be mapped to the same point in $\Rd$ without ``ruining'' the representation of $\simp K$. Therefore every $d$-representable is also $d$-Matou\v sek (Proposition \ref{basic_Mat}). We begin to explore the $d$-Matou\v sek property, and how it relates to $d$-representability.

It is natural to start by restricting the general problem to low dimension, and there are three ways to do this: by fixing the dimension of the complex $\simp K$, by fixing the dimension of the dual complex $\simp K'$, or by fixing $d$ (the dimension of the space in which we look for representing convex sets). For $d=1$ representability was fully classified in 1962 by Lekkerkerker and Boland \cite{LB}. We apply their results in the context of Matou\v sek complexes, and reclassify $1$-representable complexes as $1$-Matou\v sek complexes (Theorem \ref{Mat_d=1}). For simplicial complexes $\simp K$ such that $\dim\simp K'=1$, we again obtain that $d$-Matou\v sek and $d$-representability are identical (Theorem \ref{codim_1}).

However, we describe a complex $\simp K$ such that $\dim\simp K=1$, which is $2$-Matou\v sek but not $2$-representable (see Example \ref{bad_cover}).


The structure of the remainder of this paper more or less mimics the introduction: in Section 2 we briefly summarize existing notation, definitions, and results, for simplicial complexes and the $\Zt$-index. In Section 3 we introduce our new concepts, and the main theorems utilizing them. In Section 4 we explore $d$-Matou\v sek complexes, and we conclude in Section 5 by outlining some unanswered questions.



\subsection{Acknowledgments}
This work was completed as part of the PhD thesis of the author. I would like to thank my advisor Gil Kalai for his continued guidance and support, without which this work would not be possible. I would also like to thank Pavle Blagojevi\'c for useful discussions.

\section{Preliminaries} 

In this section we choose notations and give a brief and incomplete summary of previously existing results required for our setup.



Let $V$ be a finite set. A collection $\simp K\su 2^V$ is called a \emph{simplicial complex} if it is closed under downward inclusion: whenever $F\in\simp K$ and $G\su F$, we have $G\in\simp K$. The elements of $\simp K$ are called \emph{faces} and inclusion maximal faces are called \emph{facets}. A \emph{vertex} is any $i\in V$, and we assume that $\set{i}\in \simp K$ for every $i\in V$. We also denote the vertex set of $\simp K$ by $V(\simp K)$.\\
To ease notation, we mostly use numbers and lowercase Latin letters ($1,2,\ldots,i,j,n,x$ etc.) to denote single elements, uppercase Latin letters ($I,J,F,G,V$ etc.) for sets of elements such as faces, and Greek lowercase letters ($\alpha,\beta,\gamma$ etc.) for collections of sets of elements. Simplicial complexes will mostly be denoted by sans-serif capital letters ($\simp L,\simp K,\simp G$ etc.). We also often abbreviate notation for sets described by a list of elements, for example writing $123$ instead of $\set{1,2,3}$ when possible. For every $n\in\bbN$, let $[n]=\set{1,\ldots,n}$.\\
Simplicial complexes are \emph{isomorphic} if we can obtain one from the other by renaming the vertices (and extending to faces).\\
The \emph{dimension} of a face $F\in\simp K$ equals $|F|-1$. The dimension of a simplicial complex $\simp K$ is the maximal dimension among all faces of $\simp K$.\\
For every $d\in\bbN$, the \emph{$d$-dimensional simplex} is $\Delta_{d}=2^{[d+1]}$.\\
The \emph{$k$-skeleton} of a simplicial complex $\simp K$ is the subcomplex consisting of all faces of dimension $\leq k$, and is denoted by $\simp K^{(k)}$.\\
We consider graphs as $1$-dimensional simplicial complexes: we say that a simplicial complex $\simp G\su 2^V$ is a graph if all faces of $\simp G$ have dimension $\leq 1$. The edges $E=E(\simp G)$ are the $1$-dimensional faces of $\simp G$.\\

A simplicial complex $\simp K\su 2^V$ is \emph{$d$-representable}, if there exist sets $\set{C_i}_{i\in V}$, such that each $C_i\su\Rd$ is convex, and
$$\simp K=\set{I\su V\ :\ \bigcap_{i\in I} C_i\neq\emp}.$$
Let $\set{e_v}_{v\in V}$ denote the standard basis of $\bbR^V$. For every face $F\in\simp K$, the (geometric) simplex corresponding to $F$ is 
$$\geo F=\conv \set{e_v\ :\ v\in F}$$
where $\conv$ denotes the convex hull. The geometric realization of $\simp K$ is the topological space
$$\geo{\simp K} = \bigcup_{F\in\simp K} \geo F.$$
For a point $x\in\geo{\simp K}$, the support of $x$, denoted $\supp(x)$, is the minimal face $F\in\simp K$ such that $x\in\geo F$.

For any simplicial complex $\simp K$, the \emph{barycentric subdivision} of $\simp K$ is a new simplicial complex denoted by $\sd(\simp K)$. The vertices of $\sd(\simp K)$ are $\simp K\setminus\emp$ (non-empty faces of $\simp K$), and faces in $\sd(\simp K)$ are chains of faces in $\simp K$. In other words, a collection of faces $\set{F_1,\ldots,F_n}\su\simp K$ is a face in $\sd(\simp K)$, if $\emp\subsetneq F_1\subsetneq\cdots\subsetneq F_n$.\\
It is straightforward to verify that $\geo{\simp K}$ and $\geo{\sd(\simp K)}$ are canonically homeomorphic, where the homeomorphism $\geo{\sd(\simp K)}\to\geo{\simp K}$ is defined on vertices by
$$e_F\ \mapsto\ \frac1{|F|}\sum_{i\in F}e_i\ ,$$
and extended linearly to simplices. We also note the easily verifiable fact:

\begin{lem} \label{barycentric_lemma}
Let $\simp K$ be a simplicial complex, and let $x\in\geo{\simp K}$ be a point supported by some face $F\in\simp K$. Then the point in $\geo{\sd(\simp K)}$ corresponding to $x$ is supported by $\set{F_1,\ldots,F_n}$, where $\emp\subsetneq F_1\subsetneq\cdots\subsetneq F_n=F$ is a chain of faces in $\simp K$.
\end{lem}

We often work in the category of \emph{$\Zt$-spaces}: objects (called \emph{$\Zt$-spaces}) are topological spaces with a continuous $\Zt$-action, and morphisms (called \emph{equivariant maps}) are continuous and commute with the $\Zt$-action. The action of the non-trivial element in $\Zt$ on the $\Zt$-space is called the \emph{involution}.
\newline\\
Important $\Zt$-spaces that we use include:

\begin{itemize}
\item The $d$-dimensional sphere $S^d$ (where $d$ is a non-negative integer).
\item For any topological space $X$, the deleted product of $X$ is 
$$\topdp X = \set{(x,y)\in X^2 : x\neq y}.$$
\item For any simplicial complex $\simp K$, the deleted product of $\simp K$ is the topological space:
$$\scdp{\simp K} = \set{(x,y)\in \geo{\simp K}^2 : \supp(x)\cap\supp(y) = \emp}.$$
\end{itemize}
The involution on $S^d$ is the antipodal map ($x\mapsto -x$). For the deleted products (both of a topological space and of a simplicial complex), the involution is the coordinate swap, $(x,y)\mapsto(y,x)$. We will often omit repeatedly specifying the involution when it should be clear from the context.\\
With these definitions, the Borsuk--Ulam theorem from 1933 simply asserts:

\begin{thm}[Borsuk--Ulam]
Let $d\in\mathbb{N}$. There is no equivariant map $f:S^d\to S^{d-1}$.
\end{thm}

To use the Borsuk--Ulam theorem in a setting other than spheres, we recall the definition of the $\Zt$-index:
\begin{df}[\cite{Ma}, Definition 5.3.1]
For any $\Zt$-space $A$, we define the $\Zt$-index of $A$ as $$\ind A = \min\{d\in\{0,1,2,\ldots\}\ |\ \exists \textnormal{ an equivariant } f: A\to S^d\}$$
\end{df}
This allows us to use spheres as a measure of the `size' of a $\Zt$-space, by the following immediate observations:

\begin{prop}[\cite{Ma}, Proposition 5.3.2] \label{basic_index_properties}
\ 
\begin{enumerate}
\item If $A,B$ are $\Zt$-spaces and there exists an equivariant $f:A\to B$, we \textit{must} have $\ind A\leq\ind B$ (any equivariant map $B\to S^d$ can be precomposed with $f$).
\item For every non-negative integer $d$, we have $\ind {S^d}=d$ (follows from Borsuk--Ulam).
\end{enumerate}
\end{prop}
So `large' $\Zt$-spaces (as measured by the $\Zt$-index) can't be mapped to `smaller' ones. This gives us a general framework for proving non-embeddability results: suppose we want to show that there is no continuous injective map $f:X\to Y$, for some pair of topological spaces $X$ and $Y$. Consider the deleted products $\topdp X$ and $\topdp Y$ (the former is often called the \emph{configuration space}, and the latter is the \emph{target space} of the problem). If an injective continuous map $f:X\to Y$ exists, we may apply $f$ to each coordinate of the points in $\topdp X$, to obtain a map $\topdp X\to\topdp Y$, which is clearly equivariant (an equivariant map $\topdp X\to\topdp Y$ is called a \emph{test map}). Therefore if we prove that $\ind{\topdp X}>\ind{\topdp Y}$, no test map exists, and any continuous $f:X\to Y$ must not be injective.

We note our choice to work with deleted products of simplicial complexes in this paper, rather than deleted joins. It is true that deleted joins have the advantage of being simplicial complexes themselves, and are often easier to work with (consider for example the deleted join and deleted product of $K_5$). However, the deleted product has the advantage of having dimension lower by $1$ than its deleted join counterpart, resulting in a proof of Theorem \ref{Mat_d=1} which is easier to visualize.

\comment{
\begin{rmk} 
For a simplicial complex $K$, we have $\scdp K\su \topdp{\geo K}$ but usually these spaces aren't equal, therefore $\ind {\scdp K} \le \ind {\topdp{\geo K}}$. However, it can be shown that the two spaces always have the same index.
\end{rmk}}

\section{Main setup and results} 

\subsection{Dual complexes and faithful maps}

Our first goal is to define faithful maps, and show how they relate to representability. Recall that for every simplicial complex $\simp K$ we let $\m{\simp K}$ denote the facets of $\simp K$, and defined the dual of $\simp K$, denoted $\simp K'$, as the nerve of $\m{\simp K}$  (Definition \ref{dual_complex}). We now specify some faces of $\simp K'$:

\begin{df}
Let $\simp K$ be a simplicial complex. If $i$ is a vertex of $\simp K$, let $\cont_i$ denote the set of $\simp K$-facets containing $i$:
$$\cont_i = \set{J\in \m{\simp K} : i\in J}.$$
If $F$ is a face of $\simp K$, we define:
$$\cont_F = \set{J\in \m{\simp K} : F\su J}.$$
\end{df}

Note that $\cont_F$ and $\cont_i$ are always faces of $\simp K'$. Every facet of $\simp K'$ is of the form $\cont_j$ for some vertex $j$.

The motivation behind the dual complex $\simp K'$ is to create the simplest collection of convex sets with $\simp K$ as their nerve. Clearly $\simp K$ is the nerve of the collection $\set{\cont_i}_{i\in V}$, 
so the images of $\set{\geo{\cont_i}}_{i\in V}$ under any map from $\geo{\simp K'}$ which doesn't introduce additional intersections among the images of these simplices, will still have $\simp K$ as the nerve. We call such a map \emph{faithful}:

\begin{df} \label{faithful_definition}
Let $\simp K\su 2^V$ be a simplicial complex. A \emph{faithful} map (for $\simp K$), is a continuous $f:\geo{\simp K'}\to\Rd$, such that for every $I\su V$ which satisfies 
$$\bigcap_{i\in I} f\Par{\geo{\cont_i}}\neq\emp,$$
we have $I\in\simp K$ (equivalently $\underset{i\in I}{\cap} \cont_i\neq\emp$).
\end{df}

\comment{ 
\begin{df} \label{faithful_definition}
Let $\simp K\su 2^V$ be a simplicial complex. A continuous map $f:\geo{\simp K'}\to\Rd$ is \emph{faithful}, if for every $I\su V$ which satisfies 
$$\bigcap_{i\in I} f\Par{\geo{\cont_i}}\neq\emp,$$
we have $I\in\simp K$ (equivalently $\underset{i\in I}{\cap} \cont_i\neq\emp$).
\end{df}}


\begin{thm} \label{faithful_lemma}
Let $\simp K$ be a simplicial complex. Then $\simp K$ is $d$-representable if and only if there exists a map $f:\geo{\simp K'}\to\Rd$ which is linear and faithful.
\end{thm}
\begin{proof}
W.l.o.g. assume $\simp K\su 2^{[n]}$ (we may label the vertices of $\simp K$ as $\set{1,\ldots,n}$). If $f:\geo{\simp K'}\to\Rd$ is a linear and faithful map, for every $i\in[n]$ let
$$C_i=f\Par{\geo{\cont_i}}.$$
The sets $C_1,\ldots,C_n\su\Rd$ are convex because $f$ is linear, and we claim that $\simp K$ is their nerve. Indeed, every face $I\in\simp K$ is a subset of some facet $J\in \m{\simp K}$. Thus for every $i\in I$ we have $J\in \cont_i$, so $\underset{i\in I}{\cap} \geo{\cont_i}\neq\emp$ and therefore $\underset{i\in I}{\cap} C_i\neq\emp$ as well. If $I\su [n]$ but $I\notin\simp K$, then $\underset{i\in I}{\cap} C_i=\emp$ due to $f$ being faithful.\\
Conversely, suppose that $\simp K$ is the nerve of the convex sets $C_1,\ldots,C_n\su\Rd$. For every facet $J\in \m{\simp K}$, choose some point $p_J\in\underset{i\in J}{\cap} C_i$ (this intersection is non-empty as $J$ is in the nerve). Construct $f:\geo{\simp K'}\to\Rd$ by taking $e_J\mapsto p_J$ and extending linearly. To prove that $f$ is faithful note that for every $i\in [n]$ we have
$$f\Par{\geo{\cont_i}}=\conv{\set{p_J : J\in\cont_i}}\su C_i.$$
So whenever $I\su [n]$ and $\underset{i\in I}{\cap} f\Par{\geo{\cont_i}}$ is non-empty, the intersection $\underset{i\in I}{\cap} C_i$ is also non-empty, thus $I\in\simp K$.
\end{proof}

An example illustrating the proof is given in Figure \ref{simplicial_example} below:

\begin{center}
\begin{tikzpicture}[scale=0.75]
\coordinate (v0) at (-0.7,0);
\filldraw (v0) circle (0pt) node[above] {};
\coordinate (v1) at (0+0,6);
\filldraw (v1) circle (3pt) node[above right] {1};
\coordinate (v2) at (0+0,2);
\filldraw (v2) circle (3pt) node[above right] {2};
\coordinate (v3) at (0+4,2);
\filldraw (v3) circle (3pt) node[above right] {3};
\coordinate (v4) at (0+1,4);
\filldraw (v4) circle (3pt) node[below left] {4};
\coordinate (v5) at (0+2,2);
\filldraw (v5) circle (3pt) node[above right] {5};
\coordinate (v6) at (0+2,0);
\filldraw (v6) circle (3pt) node[below] {6};
\coordinate (v7) at (0+4,0);
\filldraw (v7) circle (3pt) node[below right] {7};
\draw (v1) -- (v4);
\draw (v1) -- (v2);
\draw (v1) -- (v3);
\draw (v4) -- (v5);
\draw (v2) -- (v5);
\draw (v2) -- (v6);
\draw (v5) -- (v3);
\draw (v5) -- (v6);
\draw (v3) -- (v7);
\draw (v6) -- (v7);
\path[fill, opacity=0.35] (v2) -- (v5) -- (v6) -- cycle;
\coordinate (center1) at (7.5+2,6);
\coordinate (center2) at (7.5+0,4);
\coordinate (center3) at (7.5+2,3);
\coordinate (center4) at (7.5+4,4);
\coordinate (center5) at (7.5+2,2);
\coordinate (center6) at (7.5+0,1);
\coordinate (center7) at (7.5+1,0);
\draw[fill, opacity=0.2] (center1) ellipse (2.4 and 0.6);
\draw[fill, opacity=0.2] (center2) ellipse (0.6 and 2.4);
\draw[fill, opacity=0.2] (center3) ellipse (0.6 and 3.5);
\draw[fill, opacity=0.2] (center4) ellipse (0.6 and 2.4);
\draw[fill, opacity=0.2] (center5) ellipse (2.4 and 0.6);
\draw[fill, opacity=0.2] (center6) ellipse (0.6 and 1.5);
\draw[fill, opacity=0.2] (center7) ellipse (1.2 and 0.6);
\node at (7.5+1,6) {$C_1$};
\node at (center2) {$C_2$};
\node at (7.5+2,4) {$C_3$};
\node at (center4) {$C_4$};
\node at (7.5+1,2) {$C_5$};
\node at (center6) {$C_6$};
\node at (center7) {$C_7$};
\coordinate (center1) at (15.5+2,6);
\coordinate (center2) at (15.5+0,4);
\coordinate (center3) at (15.5+2,3);
\coordinate (center4) at (15.5+4,4);
\coordinate (center5) at (15.5+2,2);
\coordinate (center6) at (15.5+0,1);
\coordinate (center7) at (15.5+1,0);
\draw[fill, opacity=0.1] (center1) ellipse (2.4 and 0.6);
\draw[fill, opacity=0.1] (center2) ellipse (0.6 and 2.4);
\draw[fill, opacity=0.1] (center3) ellipse (0.6 and 3.5);
\draw[fill, opacity=0.1] (center4) ellipse (0.6 and 2.4);
\draw[fill, opacity=0.1] (center5) ellipse (2.4 and 0.6);
\draw[fill, opacity=0.1] (center6) ellipse (0.6 and 1.5);
\draw[fill, opacity=0.1] (center7) ellipse (1.2 and 0.6);

\coordinate (v12) at (15.5+0,6.2);
\filldraw (v12) circle (3pt) node[above left] {12};
\coordinate (v13) at (15.5+1.8,5.7);
\filldraw (v13) circle (3pt) node[below right] {13};
\coordinate (v14) at (15.5+4,6.2);
\filldraw (v14) circle (3pt) node[above right] {14};
\coordinate (v256) at (15.5-0.1,2);
\filldraw (v256) circle (3pt) node[below left] {256};
\coordinate (v35) at (15.5+2.4,2.4);
\filldraw (v35) circle (3pt) node[above right] {35};
\coordinate (v45) at (15.5+4,1.9);
\filldraw (v45) circle (3pt) node[right] {45};
\coordinate (v67) at (15.5+0.2,0);
\filldraw (v67) circle (3pt) node[left] {67};
\coordinate (v37) at (15.5+1.8,0.2);
\filldraw (v37) circle (3pt) node[right] {37};

\draw [line width=0.5mm] (v14) -- (v12);
\draw [line width=0.5mm] (v14) -- (v13);
\draw [line width=0.5mm] (v12) -- (v13);
\draw [line width=0.5mm] (v14) -- (v45);
\draw [line width=0.5mm] (v12) -- (v256);
\draw [line width=0.5mm] (v45) -- (v35);
\draw [line width=0.5mm] (v45) -- (v256);
\draw [line width=0.5mm] (v35) -- (v256);
\draw [line width=0.5mm] (v13) -- (v35);
\draw [line width=0.5mm] (v13) -- (v37);
\draw [line width=0.5mm] (v35) -- (v37);
\draw [line width=0.5mm] (v256) -- (v67);
\draw [line width=0.5mm] (v37) -- (v67);

\path[fill, opacity=0.4] (v14) -- (v12) -- (v13) -- cycle;
\path[fill, opacity=0.4] (v45) -- (v35) -- (v256) -- cycle;
\path[fill, opacity=0.4] (v13) -- (v35) -- (v37) -- cycle;
\end{tikzpicture}
\scriptsize
Figure \refstepcounter{figure_num}\label{simplicial_example}\arabic{figure_num}
\end{center}

From left to right, we have a simplicial complex $\simp K$, seven convex sets in the plane which have $\simp K$ as the nerve, and the image of $\geo{\simp K'}$ under a linear faithful map within the original convex sets.

It is also worth noting that $\geo{\simp K}$ and $\geo{\simp K'}$ are homotopy equivalent (due to the nerve theorem, see e.g. \cite[Theorem 4.4.4]{Ma}). While we don't need the to prove this equivalence, we will use the following:

\begin{lem} \label{dual_lemma}
Let $\simp K$ be a simplicial complex. There exists a map $\iota_{\simp K}:\geo{\simp K}\to\geo{\simp K'}$, such that for every $x\in\geo{\simp K}$, we have
$$\supp(x)\su\bigcap\supp\Par{\iota_{\simp K}(x)}.$$
\end{lem}

\begin{proof}
Using the canonical homeomorphism between $\geo{\simp K}$ and $\geo{\sd(\simp K)}$, we may in fact define a linear map $g:\geo{\sd(\simp K)}\to\geo{\simp K'}$. For every $\emp\neq F\in\simp K$, let $g(e_F)\in\geo{\simp K'}$ be some point with $\supp(g(e_F))=\cont_F$ (we can choose $g(e_F)$ to be the center of the simplex $\geo{\cont_F}$). This indeed extends to a linear $g:\geo{\sd(\simp K)}\to\geo{\simp K'}$: if $x\in\geo{\sd(\simp K)}$ is supported by $\set{F_1,\ldots,F_n}$ such that $\emp\subsetneq F_1\subsetneq\cdots\subsetneq F_n$, we have $\cont_{F_n}\su\ldots\su\cont_{F_1}$, therefore $g(x)$ is well defined in $\geo{\simp K'}$, and
\begin{equation} \label{g_0_prop}
\supp\Par{g(x)}\su\cont_{F_1}.
\end{equation}
By Lemma \ref{barycentric_lemma}, if $x\in\geo{\simp K}$ is supported by $F\in\simp K$, then the corresponding point in $\geo{\sd(\simp K)}$ is supported by some $\set{F_1,\ldots,F_n}$, where $\emp\subsetneq F_1\subsetneq\cdots\subsetneq F_n=F$. So (\ref{g_0_prop}) becomes
$$\supp\Par{\iota_{\simp K}(x)}\su\cont_{F_1}$$
and therefore
$$\supp(x)=F_n\su\bigcap \cont_{F_n}\su\ldots\su\bigcap \cont_{F_1}\su\bigcap\supp\Par{\iota_{\simp K}(x)}.$$
\end{proof}

\subsection{Configuration spaces and $d$-Matou\v sek complexes}

Recall Definitions \ref{configuration_space}-\ref{d-Mat}: given a simplicial complex $\simp K$, we let
$$\hat{\simp K} = \set{(x,y)\in\geo{\simp K'}^2\ :\ \Par{\bigcap\ \supp(x)}\cup\Par{\bigcap\ \supp(y)}\notin \simp K},$$
and say that $\simp K$ is $d$-Matou\v sek if there exists a continuous map $f:\hat{\simp K}\to S^{d-1}$ which is equivariant (that is, $f(y,x)=-f(x,y)$ for every $(x,y)\in\hat{\simp K}$). Equivalently, $\ind{\hat{\simp K}}\leq d-1$, where the involution on $\hat{\simp K}$ is given by $(x,y)\mapsto(y,x)$. We now prove that $\hat{\simp K}$ is the configuration space consisting of pairs that can't be identified under a faithful map:

\begin{lem} \label{configuration_lemma}
Let $\simp K$ be a simplicial complex, and suppose that $f:\geo{\simp K'}\to\Rd$ has $f(x)=f(y)$ for some $(x,y)\in\hat{\simp K}$. Then $f$ is not faithful.
\end{lem}
\begin{proof}
Let $F=\bigcap\supp(x)$ and $G=\bigcap\supp(y)$. Note that $i\in F$ implies $\supp(x)\su\cont_i$ (for every $i\in V(\simp K)$). Therefore $\supp(x)\su\bigcap_{i\in F}\cont_i$ and similarly $\supp(y)\su\bigcap_{j\in G}\cont_j$. We now deduce
$$\emp\neq f\Par{\geo{\supp(x)}}\cap f\Par{\geo{\supp(y)}}\su f\Par{\bigcap_{i\in F}\geo{\cont_i}}\cap f\Par{\bigcap_{j\in G}\geo{\cont_j}}\su \bigcap_{v\in F\cup G}f\Par{\geo{\cont_v}},$$
but $(x,y)\in\hat{\simp K}$ implies $\Par{\cap\ \supp(x)}\cup\Par{\cap\ \supp(y)}=F\cup G\notin\simp K$, so $f$ is not faithful.
\end{proof}

This implies that whenever a faithful map $\geo{\simp K'}\to\Rd$ exists (and in particular if $\simp K$ is $d$-representable), $\simp K$ is $d$-Matou\v sek:

\begin{prop} \label{faithful_Mat}
Let $d\in\bbN$, and let $\simp K$ be a simplicial complex such that there exists a faithful map $\geo{\simp K'}\to\Rd$. Then $\simp K$ is $d$-Matou\v sek.
\end{prop}

\begin{proof}
If $f:\geo{\simp K'}\to\Rd$ is faithful, define $\hat f:\hat{\simp K}\to S^{d-1}$, by
$$\hat f(x,y)=\frac{f(x)-f(y)}{\left|f(x)-f(y)\right|}$$
for every $(x,y)\in\hat{\simp K}$ (where $\left|\cdot\right|$ denotes the norm in $\Rd$). We know that $\hat f$ is well-defined from Lemma \ref{configuration_lemma}, and clearly $\hat f$ is equivariant.
\end{proof}

Theorem \ref{faithful_lemma} now implies:

\begin{prop} \label{basic_Mat}
For any $d\in\bbN$, every $d$-representable simplicial complex is $d$-Matou\v sek.
\end{prop}

\begin{rmk} \label{induced_sub}
Note that if $\simp K$ is $d$-Matou\v sek and $\simp L$ is an induced subcomplex, that is $\simp L=\simp K\cap 2^W$ for some $W\su V(\simp K)$, then $\simp L$ is also $d$-Matou\v sek. This is because every maximal face of $\simp L$ is of the form $W\cap J$, where $J\in\m{\simp K}$ is some maximal face in $\simp K$. This gives us a map $\m{\simp L}\to\m{\simp K}$, and it is straightforward to verify that by extending to $\geo{\simp L'}$ and then applying on each coordinate, we obtain an equivariant $\hat{\simp L}\to\hat{\simp K}$.
\end{rmk}

\subsection{Good covers}
Note that in Proposition \ref{faithful_Mat} we only required a faithful map $\geo{\simp K'}\to\Rd$, while $\simp K$ is also $d$-representable iff there exists a faithful map which is also \emph{linear}. This implies that the collection of $d$-Matou\v sek complexes includes a larger class of complexes which satisfy some topological generalizations of (convex) representability. Indeed, every complex which is the nerve of a good cover in $\Rd$ is $d$-Matou\v sek:

A collection of sets $U_1,U_2,\ldots,U_n\su\Rd$ is a \emph{good cover}, if every intersection of sets $\cap_{i\in I}U_i$ is either  empty or homeomorphic to an open $d$-dimensional ball. As in the convex case, the nerve of these sets is the simplicial complex whose faces are sets $I\su[n]$ for which $\cap_{i\in I}U_i\neq\emp$.


\begin{lem} \label{good_cover_faithful}
Let $\simp K$ the nerve of a good cover in $\Rd$. Then there exists a faithful map $f:\geo{\simp K'}\to\Rd$.
\end{lem}

\begin{proof}
Let $\simp K$ be the nerve of a good cover $U_1,\ldots,U_n\su\Rd$, and we will construct a faithful $f:\geo{\simp K'}\to\Rd$.\\
First, choose some $p_J\in \cap_{i\in J} U_i$ for every $J\in\m{\simp K}$, and let $f(e_J)=p_J$. Continue constructing $f$ face by face, inductively on the dimension of the faces $n$: suppose that $f$ has already been defined on every simplex $\geo{\al}$ such that $\al\in \simp K'$ is a face of dimension $<n$. Furthermore, assume our $f$ satisfies
\begin{equation} \label{faithful_eq}
f(\geo \al)\su \bigcap_{i\in\cap\al} U_i
\end{equation}
for every such $\al$. Our next task is to extend the domain of $f$ to some simplex $\geo{\sigma}$, where $\sigma\in\simp K'$ is a face of dimension $n>0$. We have already defined $f|_{\geo\al}$ for every $\al\subsetneq\sigma$, and as $\al\su\sigma$ implies $\cap\sigma\su\cap\al$, we deduce
$$f(\geo{\al})\su \bigcap_{i\in\cap\al} U_i\su \bigcap_{i\in\cap\sigma} U_i.$$
This tells us that $f(\partial\geo\sigma)\su\cap_{i\in\cap\sigma} U_i$, and since $\cap_{i\in\cap\sigma} U_i$ is contractible, we may extend $f$ to $\geo{\sigma}$ within $\cap_{i\in\cap\sigma} U_i$ as required.\\
We finally obtain a map $f:\geo{\simp K'}\to\Rd$, satisfying (\ref{faithful_eq}) for every $\al\in\simp K'$. In particular, $f(\geo{\cont_i})\su U_i$ for every $i=1,2,\ldots,n$, and we conclude that $f$ is faithful: if $I\su[n]$ and $\cap_{i\in I} f(\geo{\cont_i})\neq\emp$ then also $\cap_{i\in I} U_i\neq\emp$, thus $I\in\simp K$.
\end{proof}

\begin{cor} \label{good_cover_Mat}
Let $\simp K$ be a simplicial complex which is the nerve of a good cover in $\Rd$. Then $\simp K$ is $d$-Matou\v sek.
\end{cor}

\begin{rmk} \label{almost_good_cover}
The contractibility of all non-empty intersections is a sufficient condition allowing us to construct a faithful map for the nerve, but it isn't a necessary one. Even in the proof of Lemma \ref{good_cover_faithful}, we only needed that $\cap_{i\in\cap\sigma}U_i$ is contractible for faces $\sigma\in\simp K'$ of dimension $n>0$. Thus if $\simp K\su 2^V$ is the nerve of a collection $\set{U_i}_{i\in V}$ in $\Rd$, and $\cap_{i\in F} U_i$ is contractible for every non-maximal face $F\in\simp K$, the Lemma still holds with the same proof.
\end{rmk}
The converse of Lemma \ref{good_cover_faithful} (and of Corollary \ref{good_cover_Mat}) doesn't hold, as shown in Example \ref{bad_cover}. We can, however, classify which graphs are nerves of good covers in $\Rt$.

\subsubsection{Graphs that are nerves of good covers in $\Rt$}

We will take a short detour from our exploration of faithful maps.
\begin{df}
Let $\simp G=(V,E)$ be a graph, and let $\simp H<\simp G'$ be a subgraph of the dual of $\simp G$. For every $v\in V$, let $\simp T_v$ denote the graph induced by $\simp H$ on $\set{e\in E\ :\ v\in e}\su V(\simp G')$.\\
We say that $\simp H$ is \emph{a model} (for $\simp G$), if every $\simp T_v$ is a tree.
\end{df}

\begin{thm}
Let $\simp G=(V,E)$ be a graph. Then $\simp G$ is the nerve of a good cover in $\Rt$ if and only if there exists a \emph{planar} model for $\simp G$.
\end{thm}

This gives us an algorithm that determines which graphs are nerves of good covers in $\Rt$: simply check the planarity of every model. In contrast, the general problem of recognizing good covers is undecidable (see \cite{TT}, as well as \cite{FWZ} for the related problem of embeddability).

\begin{proof}
Let $\simp H$ be a planar model for $\simp G$, and let $f:\geo{\simp H}\to\Rt$ be an embedding. We may assume that $f$ is linear. For every $v\in V$, we know that $\simp T_v<\simp H$ is a tree, and note that for any $v,w\in V$, the intersection $f(\simp T_v)\cap f(\simp T_w)$ consists of the single point $f(e)\in\Rt$ if $e=\set{v,w}\in E$, and is empty otherwise (because $f$ is an embedding).\\
This allows us to define $U_v\su\Rt$ as the $\varepsilon$ neighborhood of $f(\simp T_v)$, where $\varepsilon>0$ is sufficiently small, so that $\set{U_v}_{v\in V}$ is a good cover and $U_v\cap U_w$ is empty when $v\neq w$ and $\set{v,w}\notin E$.

On the other hand, suppose $\simp G$ is the nerve of a good cover $\calU=\set{U_v}_{v\in V}$ in $\Rt$. We may assume that every set $U_v$ (and therefore every non-empty intersection among sets from $\calU$) is bounded by a piecewise linear closed simple curve. 
We will consider part of the adjacency graph of the regions bounded by these curves. For every $x\in\Rt$, let $c(x)=\set{v\in V: x\in U_v}$ denote the \emph{color} of $x$. Define a region in $R\su\Rt$ to be a maximal connected open set, with constant color, which we denote $c(R)$ (see Figure \ref{good_cover_example}).
\begin{center}
\refstepcounter{figure_num}\label{good_cover_example}
\begin{tikzpicture}[scale=0.95]

\coordinate (p11) at (0,0);
\coordinate (p12) at (12,0);
\coordinate (p13) at (12,3);
\coordinate (p14) at (0,3);%
\path[fill, opacity=0.2] (p11) to (p12) to (p13) to (p14) to cycle;
\draw (p11) -- (p12) -- (p13) -- (p14) -- cycle;
\node (U1a) at (0.875,0.5) {$U_1$};
\node (U1b) at (4.5,1.5) {$U_1$};
\node (U1c) at (8.125,0.5) {$U_1$};
\node (U1d) at (11,2.5) {$U_1$};

\coordinate (p21) at (1,4);
\coordinate (p22) at (1.16,3.2);
\coordinate (p23) at (1.84,-0.2);
\coordinate (p24) at (2,-1);
\coordinate (p25) at (7,-1);
\coordinate (p26) at (7.16,-0.2);
\coordinate (p27) at (7.84,3.2);
\coordinate (p28) at (8,4);
\coordinate (p29) at (7,4);
\coordinate (p210) at (6.733,3.2);
\coordinate (p211) at (6,1);
\coordinate (p212) at (3,1);
\coordinate (p213) at (2.267,3.2);
\coordinate (p214) at (2,4);%
\path[fill, opacity=0.2] (p21) to (p22) to (p23) to (p24) to (p25) to (p26) to (p27) to (p28) to (p29) to (p210) to (p211) to (p212) to (p213) to (p214) to cycle;
\draw[dashed] (p22) -- (p21);
\draw (p22) -- (p23);
\draw[dashed] (p24) -- (p23);
\draw[dashed] (p25) -- (p26);
\draw (p26) -- (p27);
\draw[dashed] (p27) -- (p28);
\draw[dashed] (p210) -- (p29);
\draw (p210) -- (p211) -- (p212) -- (p213);
\draw[dashed] (p213) -- (p214);
\node (U2a) at (1.65,3.5) {$U_2$};
\node (U2b) at (4.5,-0.5) {$U_2$};
\node (U2c) at (7.35,3.5) {$U_2$};
\node (U12) at (4.5,0.5) {$U_1\cap U_2$};

\coordinate (p31) at (-1,2);
\coordinate (p32) at (-0.2,2);
\coordinate (p33) at (1.2,2);
\coordinate (p34) at (1.2,1);
\coordinate (p35) at (-0.2,1);
\coordinate (p36) at (-1,1);%
\path[fill, opacity=0.2] (p31) to (p32) to (p33) to (p34) to (p35) to (p36) to cycle;
\draw[dashed] (p32) -- (p31);
\draw (p32) -- (p33) -- (p34) -- (p35);
\draw[dashed] (p35) -- (p36);
\node (U3) at (-0.5,1.5) {$U_3$};
\node (U13) at (0.6,1.5) {\scriptsize{$U_1\cap U_3$}};

\coordinate (p41) at (3,4);
\coordinate (p42) at (3,3.2);
\coordinate (p43) at (3,2);
\coordinate (p44) at (4.2,2);
\coordinate (p45) at (4.2,3.2);
\coordinate (p46) at (4.2,4);%
\path[fill, opacity=0.2] (p41) to (p42) to (p43) to (p44) to (p45) to (p46) to cycle;
\draw[dashed] (p42) -- (p41);
\draw (p42) -- (p43) -- (p44) -- (p45);
\draw[dashed] (p45) -- (p46);
\node (U4) at (3.6,3.5) {$U_4$};
\node (U14) at (3.6,2.5) {\scriptsize{$U_1\cap U_4$}};

\coordinate (p51) at (4.8,4);
\coordinate (p52) at (4.8,3.2);
\coordinate (p53) at (4.8,2);
\coordinate (p54) at (6,2);
\coordinate (p55) at (6,3.2);
\coordinate (p56) at (6,4);%
\path[fill, opacity=0.2] (p51) to (p52) to (p53) to (p54) to (p55) to (p56) to cycle;
\draw[dashed] (p52) -- (p51);
\draw (p52) -- (p53) -- (p54) -- (p55);
\draw[dashed] (p55) -- (p56);
\node (U5) at (5.4,3.5) {$U_5$};
\node (U15) at (5.4,2.5) {\scriptsize{$U_1\cap U_5$}};

\coordinate (p61) at (8.9,4);
\coordinate (p62) at (8.9,3.2);
\coordinate (p63) at (8.9,-0.2);
\coordinate (p64) at (8.9,-1);
\coordinate (p65) at (10.1,-1);
\coordinate (p66) at (10.1,-0.2);
\coordinate (p67) at (10.1,3.2);
\coordinate (p68) at (10.1,4);%
\path[fill, opacity=0.2] (p61) to (p62) to (p63) to (p64) to (p65) to (p66) to (p67) to (p68) to cycle;
\draw[dashed] (p62) -- (p61);
\draw (p62) -- (p63);
\draw[dashed] (p63) -- (p64);
\draw[dashed] (p66) -- (p65);
\draw (p66) -- (p67);
\draw[dashed] (p67) -- (p68);
\node (U6a) at (9.5,3.5) {$U_6$};
\node (U6b) at (9.5,-0.5) {$U_6$};
\node (U16) at (9.5,1.5) {\scriptsize{$U_1\cap U_6$}};

\coordinate (p71) at (13,2);
\coordinate (p72) at (12.2,2);
\coordinate (p73) at (10.8,2);
\coordinate (p74) at (10.8,1);
\coordinate (p75) at (12.2,1);
\coordinate (p76) at (13,1);%
\path[fill, opacity=0.2] (p71) to (p72) to (p73) to (p74) to (p75) to (p76) to cycle;
\draw[dashed] (p72) -- (p71);
\draw (p72) -- (p73) -- (p74) -- (p75);
\draw[dashed] (p75) -- (p76);
\node (U7) at (12.5,1.5) {$U_7$};
\node (U17) at (11.4,1.5) {\scriptsize{$U_1\cap U_7$}};

\end{tikzpicture}
\scriptsize\smallskip\\
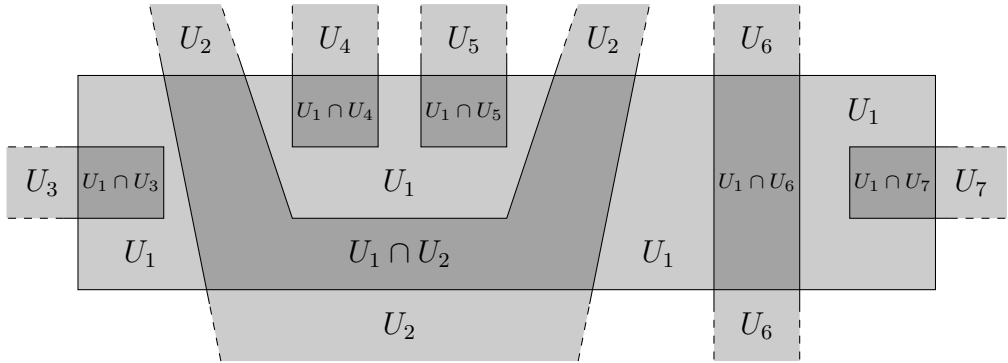
Figure \arabic{figure_num}: an example of an $\varepsilon$-neighborhood of $U_1\in\calU$, with labeled regions.
\medskip
\end{center}
Let $\simp H_0=\simp H_0\Par{\calU}$ be the adjacency graph of all regions whose color is not $\emp$. For every $v\in V$, let $\simp H_{0,v}=\simp H_{0,v}\Par{\calU}$ be the graph induced by $\simp H_0$ on all regions $R$ such that $v\in c(R)$ (see Figure \ref{initial_cover_graph}).  
We now claim:
\begin{enumerate}
\item For every $e=\set{v,w}\in E$ there is a unique region $R_e\in V(\simp H_0)$ colored by $e$, this region is only adjacent to regions whose color is $\set{v}$ or $\set{w}$, and the number of $\set{v}$-colored regions adjacent to $R_e$ is uniquely determined by $U_v$ and $U_w$.
\item For every $v\in V$, the graph $\simp H_{0,v}$ is a tree.
\end{enumerate}

\begin{center}
\refstepcounter{figure_num}\label{initial_cover_graph}
\scriptsize
\begin{tikzpicture}[scale=0.8]

\coordinate (p11) at (0,0);
\coordinate (p12) at (12,0);
\coordinate (p13) at (12,3);
\coordinate (p14) at (0,3);
\coordinate (U1a) at (0.875,0.5);
\filldraw[violet] (U1a) circle (2pt) node[black, below left] {1};
\coordinate (U1b) at (4.5,1.5);
\filldraw[violet] (U1b) circle (2pt) node[black, below right] {1};
\coordinate (U1c) at (8.125,0.5);
\filldraw[violet] (U1c) circle (2pt) node[black, below right] {1};
\coordinate (U1d) at (11,2.5);
\filldraw[violet] (U1d) circle (2pt) node[black, above] {1};
\path[fill, opacity=0.1] (p11) to (p12) to (p13) to (p14) to cycle;

\coordinate (p21) at (1,4);
\coordinate (p22) at (2,-1);
\coordinate (p23) at (7,-1);
\coordinate (p24) at (8,4);
\coordinate (p25) at (7,4);
\coordinate (p26) at (6,1);
\coordinate (p27) at (3,1);
\coordinate (p28) at (2,4);
\coordinate (U2a) at (1.65,3.5);
\filldraw (U2a) circle (2pt) node[above right] {2};
\coordinate (U2b) at (4.5,-0.5);
\filldraw (U2b) circle (2pt) node[right] {2};
\coordinate (U2c) at (7.35,3.5);
\filldraw (U2c) circle (2pt) node[above left] {2};
\coordinate (U12) at (4.5,0.5);
\filldraw[violet] (U12) circle (2pt) node[black, below right] {12};
\path[fill, opacity=0.1] (p21) to (p22) to (p23) to (p24) to (p25) to (p26) to (p27) to (p28) to cycle;

\coordinate (p31) at (-1,2);
\coordinate (p32) at (1.2,2);
\coordinate (p33) at (1.2,1);
\coordinate (p34) at (-1,1);
\coordinate (U3) at (-0.5,1.5);
\filldraw (U3) circle (2pt) node[below] {3};
\coordinate (U13) at (0.6,1.5);
\filldraw[violet] (U13) circle (2pt) node[black, above right] {13};
\path[fill, opacity=0.1] (p31) to (p32) to (p33) to (p34) to cycle;

\coordinate (p41) at (3,4);
\coordinate (p42) at (3,2);
\coordinate (p43) at (4.2,2);
\coordinate (p44) at (4.2,4);
\coordinate (U4) at (3.6,3.5);
\filldraw (U4) circle (2pt) node[left] {4};
\coordinate (U14) at (3.6,2.5);
\filldraw[violet] (U14) circle (2pt) node[black, left] {14};
\path[fill, opacity=0.1] (p41) to (p42) to (p43) to (p44) to cycle;

\coordinate (p51) at (4.8,4);
\coordinate (p52) at (4.8,2);
\coordinate (p53) at (6,2);
\coordinate (p54) at (6,4);
\coordinate (U5) at (5.4,3.5);
\filldraw (U5) circle (2pt) node[right] {5};
\coordinate (U15) at (5.4,2.5);
\filldraw[violet] (U15) circle (2pt) node[black, right] {15};
\path[fill, opacity=0.1] (p51) to (p52) to (p53) to (p54) to cycle;

\coordinate (p61) at (8.9,4);
\coordinate (p62) at (8.9,-1);
\coordinate (p63) at (10.1,-1);
\coordinate (p64) at (10.1,4);
\coordinate (U6a) at (9.5,3.5);
\filldraw (U6a) circle (2pt) node[left] {6};
\coordinate (U6b) at (9.5,-0.5);
\filldraw (U6b) circle (2pt) node[right] {6};
\coordinate (U16) at (9.5,1.5);
\filldraw[violet] (U16) circle (2pt) node[black, above left] {16};
\path[fill, opacity=0.1] (p61) to (p62) to (p63) to (p64) to cycle;

\coordinate (p71) at (13,2);
\coordinate (p72) at (10.8,2);
\coordinate (p73) at (10.8,1);
\coordinate (p74) at (13,1);
\coordinate (U7) at (12.5,1.5);
\filldraw (U7) circle (2pt) node[below] {7};
\coordinate (U17) at (11.4,1.5);
\filldraw[violet] (U17) circle (2pt) node[black, below left] {17};
\path[fill, opacity=0.1] (p71) to (p72) to (p73) to (p74) to cycle;

\draw (U3) -- (U13);
\draw[very thick, violet] (U13) -- (U1a) -- (U12) -- (U1c) -- (U16) -- (U1d) -- (U17);
\draw (U17) -- (U7);
\draw (U2b) -- (U12);
\draw[very thick, violet] (U12) -- (U1b) -- (U14);
\draw (U14) -- (U4);
\draw[very thick, violet] (U1b) -- (U15);
\draw (U15) -- (U5);
\draw (U6b) -- (U16) -- (U6a);
\draw (U2a) .. controls (2,0.8) .. (U12) .. controls (7,0.8) .. (U2c);
\draw[dashed] (U3) -- (-1,1.5);
\draw[dashed] (U2a) -- (1.5,4);
\draw[dashed] (U4) -- (3.6,4);
\draw[dashed] (U2b) -- (4.5,-1);
\draw[dashed] (U5) -- (5.4,4);
\draw[dashed] (U2c) -- (7.5,4);
\draw[dashed] (U6a) -- (9.5,4);
\draw[dashed] (U6b) -- (9.5,-1);
\draw[dashed] (U7) -- (13,1.5);

\end{tikzpicture}
\scriptsize\smallskip\\
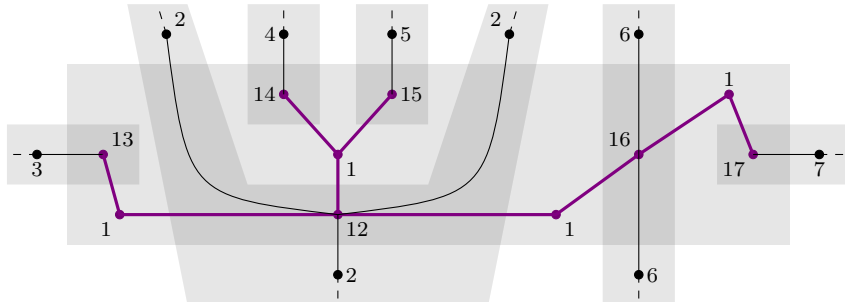
Figure \arabic{figure_num}: part of the graph $\simp H_0$ obtained from Figure \ref{good_cover_example}. The subgraph $\simp H_{0,1}$ is emphasized in violet.
\medskip
\end{center}

(1) is straightforward. 
To prove (2) note that $\simp H_{0,v}$ is connected: suppose that $R,R'\su U_v$ are two regions. The set $U_v$ is contractible and therefore path-connected, so we may consider a path $[0,1]\to U_v$ which connects some point in $R$ to a point in $R'$. This path must pass through a finite sequence of adjacent regions $R=R_1,\ldots,R_n=R'$, such that $v\in C(R_i)$ for all $i$.\\
It is now sufficient to show that $|E\Par{\simp H_{0,v}}|=|V\Par{\simp H_{0,v}}|-1$. If $\deg_{\simp G}\Par{v}=0$ then $\simp H_{0,v}$ is a single point. So assume there exists $u\in V$ such that $e=\set{v,u}\in E$. Let $k$ be the number of $\set{v}$-colored regions adjacent to $R_e$, and let $\simp H_{0,v}^-=\simp H_{0,v}\Par{\calU-\set{U_u}}$, noting that $\calU-\set{U_u}$ is also a good cover. Our claim follows by induction on $\deg_{\simp G}(v)$ if we prove that $\abs{V\Par{\simp H_{0,v}}}=\abs{V\Par{\simp H_{0,v}^-}}+k$ and $\abs{E\Par{\simp H_{0,v}}}=\abs{E\Par{\simp H_{0,v}^-}}+k$. Indeed,
$$\set{R_e}\cup\set{R: \text{$R$ is a $\set{v}$-colored region adjacent to $R_e$}}$$
is a collection of $k+1$ regions with $k$ adjacencies among them in $\simp H_{0,v}$, which unite to become a single region in $\simp H_{0,v}^-$. It follows from (1) that no other regions or adjacencies in $\simp H_{0,v}$ are affected by removing $U_u$.

We can now conclude our proof: for every region $R\in V\Par{\simp H_0}$ colored by $\set{v}$ for some $v\in V$, choose an arbitrary adjacent $R'$ (colored by some $e=\set{v,w}$), and contract $\set{R,R'}$, identifying the new region as $R'$. Let $\simp H$ be the graph obtained by preforming such contractions until the only remaining regions are colored by $E$ (see Figure \ref{final_cover_graph}).\\

\begin{center}
\refstepcounter{figure_num}\label{final_cover_graph}
\scriptsize
\begin{tikzpicture}[scale=0.8]

\coordinate (p11) at (0,0);
\coordinate (p12) at (12,0);
\coordinate (p13) at (12,3);
\coordinate (p14) at (0,3);
\coordinate (U1a) at (0.875,0.5);
\coordinate (U1b) at (4.5,1.5);
\coordinate (U1c) at (8.125,0.5);
\coordinate (U1d) at (11,2.5);

\coordinate (p21) at (1,4);
\coordinate (p22) at (2,-1);
\coordinate (p23) at (7,-1);
\coordinate (p24) at (8,4);
\coordinate (p25) at (7,4);
\coordinate (p26) at (6,1);
\coordinate (p27) at (3,1);
\coordinate (p28) at (2,4);
\coordinate (U2a) at (1.65,3.5);
\coordinate (U2b) at (4.5,-0.5);
\coordinate (U2c) at (7.35,3.5);
\coordinate (U12) at (4.5,0.5);
\filldraw (U12) circle (2pt) node[below right] {12};

\coordinate (p31) at (-1,2);
\coordinate (p32) at (1.2,2);
\coordinate (p33) at (1.2,1);
\coordinate (p34) at (-1,1);
\coordinate (U3) at (-0.5,1.5);
\coordinate (U13) at (0.6,1.5);
\filldraw (U13) circle (2pt) node[above right] {13};

\coordinate (p41) at (3,4);
\coordinate (p42) at (3,2);
\coordinate (p43) at (4.2,2);
\coordinate (p44) at (4.2,4);
\coordinate (U4) at (3.6,3.5);
\coordinate (U14) at (3.6,2.5);
\filldraw (U14) circle (2pt) node[left] {14};

\coordinate (p51) at (4.8,4);
\coordinate (p52) at (4.8,2);
\coordinate (p53) at (6,2);
\coordinate (p54) at (6,4);
\coordinate (U5) at (5.4,3.5);
\coordinate (U15) at (5.4,2.5);
\filldraw (U15) circle (2pt) node[right] {15};

\coordinate (p61) at (8.9,4);
\coordinate (p62) at (8.9,-1);
\coordinate (p63) at (10.1,-1);
\coordinate (p64) at (10.1,4);
\coordinate (U6a) at (9.5,3.5);
\coordinate (U6b) at (9.5,-0.5);
\coordinate (U16) at (9.5,1.5);
\filldraw (U16) circle (2pt) node[above left] {16};

\coordinate (p71) at (13,2);
\coordinate (p72) at (10.8,2);
\coordinate (p73) at (10.8,1);
\coordinate (p74) at (13,1);
\coordinate (U7) at (12.5,1.5);
\coordinate (U17) at (11.4,1.5);
\filldraw (U17) circle (2pt) node[below left] {17};

\draw (U3) -- (U13);
\draw[very thick, violet] (U13) .. controls (U1a) .. (U12) .. controls (U1c) .. (U16) .. controls (U1d) .. (U17);
\draw (U17) -- (U7);
\draw (U2b) -- (U12);
\draw[very thick, violet] (U12) .. controls (U1b) .. (U14);
\draw (U14) -- (U4);
\draw[very thick, violet] (U14) .. controls (4.5, 1.7) .. (U15);
\draw (U15) -- (U5);
\draw (U6b) -- (U16) -- (U6a);
\draw (U2a) .. controls (2,0.8) .. (U12) .. controls (7,0.8) .. (U2c);
\draw[dashed] (U3) -- (-1,1.5);
\draw[dashed] (U2a) -- (1.5,4);
\draw[dashed] (U4) -- (3.6,4);
\draw[dashed] (U2b) -- (4.5,-1);
\draw[dashed] (U5) -- (5.4,4);
\draw[dashed] (U2c) -- (7.5,4);
\draw[dashed] (U6a) -- (9.5,4);
\draw[dashed] (U6b) -- (9.5,-1);
\draw[dashed] (U7) -- (13,1.5);

\end{tikzpicture}
\scriptsize\smallskip\\
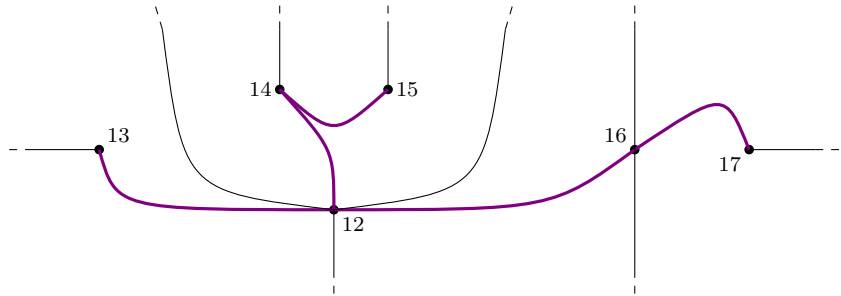
Figure \arabic{figure_num}: part of the graph $\simp H$ obtained from $\simp H_0$ given in Figure \ref{initial_cover_graph}. The subgraph $\simp T_1$ is emphasized in violet.
\medskip
\end{center}

Then $\simp H$ is a planar model for $\simp G$: we know that $\simp H$ is planar as it was obtained from $\simp H_0$ by contracting edges. Every $e\in E$ corresponds to a unique vertex in $\simp H$, and for every $v\in V$ the graph $\simp T_v$ was obtained from the tree $\simp H_{0,v}$ by contracting edges, therefore $\simp T_v$ is also a tree.
\end{proof}

\subsection{Proving and disproving $d$-representability}
As Theorem \ref{faithful_lemma} implies, there is a close connection between $d$-representability of $\simp K$ and our ability to linearly embed 
(or ``approximately'' embed) $\simp K'$ in $\Rd$. Indeed, we have

\begin{thm} \label{codim_rep}
Let $\simp K$ be a simplicial complex, such that there exists a continuous injective map $f:\geo{\simp K'}\to\Rd$. Then $\simp K$ is the nerve of a good cover in $\Rd$. If $f$ is also linear, $\simp K$ is $d$-representable. In particular, every $\simp K$ is $\Par{2\dim\simp K'+1}$-representable.
\end{thm}

\begin{proof}
If $f$ is injective, for every collection $\alpha_1,\ldots,\alpha_n\in\simp K'$ we have
$$\bigcap_{i=1}^n f\Par{\geo{\alpha_i}}=f\Par{\bigcap_{i=1}^n \geo{\alpha_i}}.$$
This clearly implies that $f$ is faithful, so the linear case follows from Theorem \ref{faithful_lemma}. Furthermore, every intersection among the sets $\set{f\Par{\geo{\cont_i}}}_{i\in V(\simp K)}$ is either empty or homeomorphic to a closed ball of dimension $\leq d$. For some sufficiently small $\varepsilon>0$, we may replace the sets $\set{f\Par{\geo{\cont_i}}}_{i\in V(\simp K)}$ with an $\varepsilon$-neighborhood of them without introducing new intersections, thus obtaining our required good cover.\\
We conclude by proving that whenever $\simp L$ is a simplicial complex and $d=2\dim\simp L+1$, there exists a linear injective map $f:\geo{\simp L}\to\Rd$. Indeed, every linear $f$ taking the vertices of $\simp L$ to points in general position will suffice (for instance, we may map the vertices of $\simp L$ to distinct points on the moment curve, $\set{\Par{t,t^2,\ldots,t^d}\ :\ t\in\bbR}$), see \cite[Theorem 1.6.1]{Ma}.
\end{proof}

\comment{
\begin{rmk}
If the continuous map $\geo{\simp K'}\to\Rd$ is injective but not linear, $\simp K$ is the nerve of a good cover in $\Rd$.
\end{rmk}}

Of course the converse doesn't hold - some faithful maps are not injective (see Figure \ref{simplicial_example}). It is therefore insufficient (in the general case) to prove that $\geo{\simp K'}$ doesn't embed in $\Rd$, if our goal is to prove that $\simp K$ is not $d$-representable. We can however prove:


\comment{
\begin{thm}
Let $\simp K,\simp L$ be simplicial complexes such that $\simp K'\cong\simp L$, and the under the same isomorphism $\hat{\simp K}\cong\scdp{\simp L}$. Suppose that $\simp L$ has the property: for every continuous $\geo{\simp L}\to\Rd$, there exist in $\simp L$ two disjoint faces whose images intersect. Then $\simp K$ is not $d$-representable.
\end{thm} }

\begin{thm} \label{non_rep}
Let $\simp L$ be a simplicial complex with the property: for every continuous ${\geo{\simp L}\to\Rd}$ there exist in $\simp L$ two disjoint faces whose images in $\Rd$ intersect.\\
Let $\simp K$ be a simplicial complex, and suppose that there exists a continuous ${\phi:\geo{\simp L}\to\geo{\simp K'}}$, which extends to $\scdp{\simp L}\to\hat{\simp K}$ (by applying $\phi$ on each coordinate). Then $\simp K$ is not $d$-representable (and in fact not even the nerve of a good cover in $\Rd$).
\end{thm} 

\begin{proof}

Assume by contradiction that $\simp K$ is $d$-representable or the nerve of a good cover in $\Rd$. In particular there exists a faithful $f:\geo{\simp K'}\to\Rd$ 
(by Theorem \ref{faithful_lemma} or Lemma \ref{good_cover_faithful}).
Of course $f\circ\phi:\geo{\simp L}\to\Rd$ is continuous, so there exist $a,b\in\geo{\simp L}$ supported by disjoint faces, such that $f(\phi(a))=f(\phi(b))$. But now $(a,b)\in\scdp{\simp L}$, implying $\Par{\phi(a),\phi(b)}\in\hat{\simp K}$ and contradicting Lemma \ref{configuration_lemma}.
\end{proof}

\begin{rmk}
Theorem \ref{non_rep} is applicable in a wide variety of cases. In some cases, $\hat{\simp K}$ is precisely the deleted product of $\simp L=\simp K'$, so $\phi$ is just the identity map (for example when $\simp K'$ is a graph, see Theorem \ref{codim_1}). In Lemma \ref{dual_classifying} we prove that every $\simp L$ has some $\simp K$ for which the theorem applies, so every such non-embeddability result translates to a non-representability result for some $\simp K$.\\
In another example due to Tancer \cite{TaRep}, $\simp L$ is the $d$-skeleton of the $\Par{2d+2}$-dimensional simplex, and $\simp K$ is the barycentric subdivision of $\simp L$. The Van Kampen-Flores Theorem 
states that any continuous $\geo{\simp L}\to\bbR^{2d}$ has some two disjoint $\simp L$-faces whose images intersect. Tancer's proof that $\simp K$ is not $2d$-representable is essentially equivalent to an application of Theorem \ref{non_rep} with an implicit construction of $\phi:\geo{\simp L}\to\geo{\simp K'}$.
\end{rmk}

We now want to show that the value $\Par{2\dim\simp K'+1}$ in Theorem \ref{codim_rep} is the best possible: there exist complexes $\simp K$ such that $\simp K'$ is $d$-dimensional, but $\simp K$ is not $2d$-representable. We know (from Van Kampen-Flores) of a suitable $\simp L$, so to apply Theorem \ref{non_rep} we prove:


\begin{lem} \label{dual_classifying}
Let $\simp L$ be a simplicial complex. Then there exists a simplicial complex $\simp K$ and an isomorphism $\phi:\simp L\to\simp K'$, which extends (by applying $\phi$ on each coordinate) to an equivariant homeomorphism $\scdp{\simp L}\to\hat{\simp K}$.
\end{lem}

\begin{proof}
Assume that $V(\simp L)=\int n$. Let $\simp K$ be the nerve of $\simp L$. Explicitly,
$$\simp K=\set{\al\su\simp L\ :\ \bigcap\al\neq\emp}.$$
For every $i\in\int n$ denote
$$\cont_i=\set{F\in\simp L\ :\ i\in F},$$
and for every $I\su\int n$ denote
$$\calC_I=\set{\cont_i\ :\ i\in I}.$$
Clearly $\m{\simp K}=\set{\cont_1,\ldots,\cont_n}$. Our isomorphism is obtained by applying $i\mapsto\cont_i$ to map the vertices. To prove that this induces an isomorphism $\simp L\to\simp K'$, we must show that every $I\su\int n$ has $I\in \simp L$ if and only if $\calC_I\in \simp K'$.

Indeed, if $I\in\simp L$ then $I\in\cont_i$ for every $i\in I$, thus $\bigcap_{i\in I}\cont_i\neq\emp$ proving $\calC_I\in\simp K'$. On the other hand, if $\calC_I\in\simp K'$ then $\cap\calC_I\neq\emp$, so let $F\in\bigcap_{i\in I}\cont_i$. In particular $F\in\simp L$. Furthermore, $F\in\cont_i$ implies $i\in F$, and this holds for every $i\in I$, proving $I\su F$ and therefore $I\in\simp L$.

We prove the remaining claim on the configuration spaces by showing that for every $F,G\in\simp L$, we have
$$F\cap G\neq\emp \quad\iff\quad \Par{\bigcap \calC_F}\cup\Par{\bigcap \calC_G}\in\simp K\ .$$
Suppose that $x\in F\cap G$. Then $\bigcap\calC_F=\bigcap_{i\in F}\cont_i\su\cont_x$, and similarly $\bigcap\calC_G\su\cont_x$, thus $\Par{\bigcap \calC_F}\cup\Par{\bigcap \calC_G}\su\cont_x\in\simp K$. On the other hand, suppose that $\al:=\Par{\bigcap \calC_F}\cup\Par{\bigcap \calC_G}\in\simp K$. For every $i\in F$ we have $F\in\cont_i$, thus $F\in\Par{\bigcap\calC_F}$ and $F\in\al$. Similarly $G\in\al$. Note that $\al\in\simp K$ implies $\bigcap\al\neq\emp$, and we deduce $\emp\neq\bigcap\al\su F\cap G$.
\end{proof}

\begin{cor}
For every $d\in\bbN$, there exist simplicial complexes $\simp K$ such that $\dim\simp K'=d$, and $\simp K$ is not $2d$-representable.
\end{cor}

\begin{proof}
Let $\simp L$ be the $d$-skeleton of the $\Par{2d+2}$-dimensional simplex, and apply Lemma \ref{dual_classifying} on $\simp L$ to obtain a simplicial complex $\simp K$ and corresponding $\phi:\geo{\simp L}\to\geo{\simp K'}$. The Van Kampen-Flores Theorem (see \cite[Theorem 5.1.1]{Ma}) states that any continuous $\geo{\simp L}\to\bbR^{2d}$ has some two disjoint $\simp L$-faces whose images intersect. We deduce from Theorem \ref{non_rep} that $\simp K$ is not $2d$-representable.
\end{proof}

\section{Results on $d$-Matou\v sek complexes}  

We now turn our attention to the $d$-Matou\v sek property, and begin to explore some cases where it fully captures $d$-representability, and others where it fails to do so.

In the classification of $1$-representable complexes, Lekkerkerker and Boland \cite{LB} introduced the notion of \emph{asteroidal triples} in graphs, and proved that $1$-skeletons of a $1$-representable complexes are precisely graphs which have no asteroidal triple and no induced cycle of length $\geq 4$. Before proving that every $1$-representable complex is $1$-Matou\v sek, we extend the notion of an asteroidal triple.

\subsection{Asteroidal maps}

Recall that $\Delta_{d+1}^{(d)}$ denotes the $d$-dimensional skeleton of a $(d+1)$-dimensional simplex.
\begin{df}
Let $\simp K$ be a simplicial complex, and let $d\in\bbN$. A \emph{$d$-asteroidal} map into $\simp K$ is a continuous $f:\geo{\Delta_{d+1}^{(d)}}\to\geo{\simp K}$, such that for every $x,y\in\geo{\Delta_{d+1}^{(d)}}$ with $\supp(x)\cap\supp(y)=\emp$, we have 
$$\supp\Par{f(x)}\cup\supp\Par{f(y)}\notin\simp K\ .$$
\end{df}

To see how this generalizes the notion of asteroidal triples (when $d=1$), note that any $3$ vertices form an asteroidal triple (as defined in \cite{LB}) if they are the image of the $3$ vertices of $\Delta_{2}^{(1)}$ under a $1$-asteroidal map. By allowing the vertices of $\Delta_{2}^{(1)}$ to be mapped to general points in $\geo{\simp K}$ (not necessarily vertices), we find that a $1$-asteroidal map into $\simp K$ exists whenever $\simp K$ is not $1$-representable.

Additionally, we obtain a $d$-dimensional analogue of asteroidal triples, and a property which implies that a complex is not $d$-representable: we can define $d+2$ points in $\geo{\simp K}$ to be in \emph{asteroidal position} if they are the image of the vertices of $\Delta_{d+1}^{(d)}$ under a $d$-asteroidal map. We now prove that any complex with $d+2$ points in asteroidal position is not $d$-Matou\v sek (and therefore by Proposition \ref{basic_Mat} not $d$-representable):

\begin{prop} \label{asteroidal_Mat}
Let $f:\geo{\Delta_{d+1}^{(d)}}\to\geo{\simp K}$ be a $d$-asteroidal map, for some simplicial complex $\simp K$ and $d\in\bbN$. Then $\simp K$ is not $d$-Matou\v sek.
\end{prop}

\begin{proof}
Denote $\simp L=\Delta_{d+1}^{(d)}$, and define $\phi:=\iota_{\simp K}\circ f:\geo{\simp L}\to\geo{\simp K'}$, where $\iota_{\simp K}:\geo{\simp K}\to\geo{\simp K'}$ is the map from Lemma \ref{dual_lemma}, which satisfies $\supp(z)\su\bigcap\supp\Par{\iota_{\simp K}(z)}$ for every $z\in\geo{\simp K}$. 
This implies that for every $x,y\in\geo{\simp L}$ with disjoint supports, we have
$$\simp K\not\ni\supp(f(x))\cup\supp(f(y))\su \Par{\bigcap\supp(\phi(x))}\cup\Par{\bigcap\supp(\phi(y))},$$
so by applying $\phi$ on each coordinate, we obtain an equivariant map $\scdp{\simp L}\to\hat{\simp K}$. Our claim now follows from Proposition \ref{basic_index_properties} and the fact that
$$\scdp{\simp L}=\scdp{\Par{\Delta_{d+1}}}\cong S^d$$
as a $\Zt$-space.
\end{proof}
\begin{rmk}
Note that we may also prove that an asteroidal map into $\simp K$ implies that $\simp K$ is not $d$-representable by using the more general statement given in Theorem \ref{non_rep}. We have our complex $\simp L=\Delta_{d+1}^{(d)}$ and a continuous $\phi=\iota_{\simp K}\circ f:\geo{\simp L}\to\geo{\simp K'}$. The fact that every continuous $\geo{\simp L}\to\Rd$ identifies some pair from disjoint faces is simply the topological version of Radon's theorem \cite{BB}.
\end{rmk}

\subsection{$1$-representable complexes}
In this subsection we prove:
\begin{thm} \label{Mat_d=1}
A simplicial complex is $1$-representable if and only if it is $1$-Matou\v sek.
\end{thm}

This gives us a topological description of $1$-representable complexes. To illustrate, consider Figure \ref{not_1rep_example} below:

\begin{center}
\refstepcounter{figure_num}\label{not_1rep_example}
\begin{tikzpicture}[scale=1.15]
\coordinate(x1) at (-4.2 + 0,0.577);
\filldraw (x1) circle (2pt) node[above right] {1};
\coordinate(x2) at (-4.2 + -0.5,-0.289);
\filldraw (x2) circle (2pt) node[above left] {2};
\coordinate(x3) at (-4.2 + 0.5,-0.289);
\filldraw (x3) circle (2pt) node[above right] {3};
\coordinate(x4) at (-4.2 + 0,1.577);
\filldraw (x4) circle (2pt) node[above] {4};
\coordinate(x5) at (-4.2 + -1.366,-0.789);
\filldraw (x5) circle (2pt) node[below left] {5};
\coordinate(x6) at (-4.2 + 1.366,-0.789);
\filldraw (x6) circle (2pt) node[below right] {6};

\draw (x1) -- (x2);
\draw (x1) -- (x3);
\draw (x1) -- (x4);
\draw (x2) -- (x3);
\draw (x2) -- (x5);
\draw (x3) -- (x6);

\path[fill, opacity=0.1] (x1) to (x2) to (x3) to cycle;


\coordinate(x123) at (0,0);
\filldraw (x123) circle (2pt) node[above right] {123};
\coordinate(x14) at (0,1);
\filldraw (x14) circle (2pt) node[above] {14};
\coordinate(x25) at (-0.866,-0.5);
\filldraw (x25) circle (2pt) node[below left] {25};
\coordinate(x36) at (0.866,-0.5);
\filldraw (x36) circle (2pt) node[below right] {36};

\draw (x123) -- (x14);
\draw (x123) -- (x25);
\draw (x123) -- (x36);

\scriptsize

\coordinate(r0) at (5 + 1.5,0 + 0);
\filldraw (r0) circle (2pt) node[right] {(25,36)};
\coordinate(r1) at (5 + 1.299,0 + 0.75);
\filldraw (r1) circle (2pt) node[right] {(25,123)};
\coordinate(r2) at (5 + 0.75,0 + 1.299);
\filldraw (r2) circle (2pt) node[above right] {(25,14)};
\coordinate(r3) at (5 + 0,0 + 1.5);
\filldraw (r3) circle (2pt) node[above] {(123,14)};
\coordinate(r4) at (5 + -0.75,0 + 1.299);
\filldraw (r4) circle (2pt) node[above left] {(36,14)};
\coordinate(r5) at (5 + -1.299,0 + 0.75);
\filldraw (r5) circle (2pt) node[left] {(36,123)};
\coordinate(r6) at (5 + -1.5,0 + 0);
\filldraw (r6) circle (2pt) node[left] {(36,25)};
\coordinate(r7) at (5 + -1.299,0 + -0.75);
\filldraw (r7) circle (2pt) node[left] {(123,25)};
\coordinate(r8) at (5 + -0.75,0 + -1.299);
\filldraw (r8) circle (2pt) node[below left] {(14,25)};
\coordinate(r9) at (5 + -0,0 + -1.5);
\filldraw (r9) circle (2pt) node[below] {(14,123)};
\coordinate(r10) at (5 + 0.75,0 + -1.299);
\filldraw (r10) circle (2pt) node[below right] {(14,36)};
\coordinate(r11) at (5 + 1.299,0 + -0.75);
\filldraw (r11) circle (2pt) node[right] {(123,36)};

\draw (r0) -- (r1);
\draw (r1) -- (r2);
\draw (r2) -- (r3);
\draw (r3) -- (r4);
\draw (r4) -- (r5);
\draw (r5) -- (r6);
\draw (r6) -- (r7);
\draw (r7) -- (r8);
\draw (r8) -- (r9);
\draw (r9) -- (r10);
\draw (r10) -- (r11);
\draw (r11) -- (r0);

\end{tikzpicture}
\scriptsize\\
Figure \arabic{figure_num}
\end{center}

From left to right: a simplicial complex $\simp K$ which is not $1$-representable despite being contractible, the dual complex $\simp K'$, and the configuration space $\hat{\simp K}$. The equivariant copy of $S^1$ within $\hat{\simp K}$ shows that $\simp K$ is not $1$-Matou\v sek. Our proof implies that this holds in general: for any non $1$-representable complex $\simp K$, we have an equivariant $S^1\to\hat{\simp K}$.

Note we don't give a new proof of \cite{LB}, but rely on their main result:

\begin{thm}[Lekkerkerker, Boland] \label{LB_simplified}
Let $\simp K\su 2^V$ be a simplicial complex, and let $\simp G$ be the $1$-skeleton of $\simp K$. Then $\simp K$ is $1$-representable if and only if the following three conditions are satisfied:
\begin{enumerate}
\item $\simp K$ is the clique complex of $\simp G$: whenever $I\su V$ and $\set{u,v}\in\simp G$ for all $u,v\in I$, we have $I\in\simp K$.
\item $\simp G$ has no induced cycle of length $n\geq 4$.
\item There is no $1$-asteroidal map $\geo{\Delta_2^{(1)}}\to\geo{\simp G}$ which maps vertices to vertices.
\end{enumerate}
\end{thm}

\begin{proof}
The necessity of the 1st condition is due to Helly's theorem. The remaining conditions are equivalent to \cite[Theorem 3]{LB},
noting that any $3$ vertices in $\simp G$ form an asteroidal triple (see \cite[Definition 3]{LB}) if they are the image of the $3$ vertices of $\Delta_{2}^{(1)}$ under some $1$-asteroidal map.
\end{proof}

Next, we show that the entire theorem can be rephrased in terms of asteroidal maps:

\begin{prop} \label{asteroidal_d=1}
Let $\simp K\su 2^V$ be a simplicial complex which is not $1$-representable. Then there exists a $1$-asteroidal map $\geo{\Delta_2^{(1)}}\to\geo{\simp K}$.
\end{prop}

\begin{proof}
Note that $\geo{\Delta_2^{(1)}}$ is just a triangle (without the interior), so constructing a map $\geo{\Delta_2^{(1)}}\to\geo{\simp K}$ requires specifying $3$ points in $\geo{\simp K}$ and $3$ paths connecting them (for each vertex and each edge of $\geo{\Delta_2^{(1)}}$, respectively). We will specify our vertices and paths in $\sd(\simp K)$ - that is, we will find $F_1,F_2,F_3\in\simp K$, and $\pathlet_{12}, \pathlet_{13}, \pathlet_{23}\su\simp K$, such that each $\pathlet_{jk}$ is a path in $\sd(\simp K)$ from $F_j$ to $F_k$.\\
Denote, $\pathlet_{jk}=\set{F_{jk}^0,F_{jk}^1,\ldots,F_{jk}^n}$, such that $F_{jk}^0=F_j$, $F_{jk}^n=F_k$. For every $m=1,\ldots,n$ we require $F_{jk}^{m-1}\su F_{jk}^m$ or $F_{jk}^m\su F_{jk}^{m-1}$.\\
A point in $\geo{\sd(\simp K)}$ supported by $\set{F,G}$ corresponds to a point in $\geo{\simp K}$ supported by $F\cup G$. Thus our specified $F_1,F_2,F_3,\pathlet_{12},\pathlet_{13},\pathlet_{23}$ correspond to a $1$-asteroidal map into $\simp K$, if whenever $i=1,2,3$ and $F\in\pathlet_{jk}$ such that $i\neq j,k$, we have
$$F\cup F_i\notin\simp K.$$
Since $\simp K$ is not $1$-representable, from Theorem \ref{LB_simplified} we have at least one of the following:
\begin{enumerate}
\item There exists $I\su V$ such that $\set{u,v}\in\simp K$ for all $u,v\in I$, but $I\notin\simp K$.
\item $\simp K$ has an induced cycle of length $n\geq 4$.
\item There exists a $1$-asteroidal map $\geo{\Delta_2^{(1)}}\to\geo{\simp K}$ which maps vertices to vertices.
\end{enumerate}
If (1) holds, let $I\su V$ be an inclusion-minimal set satisfying the condition. Clearly we have $|I|\geq 3$, so let $v_1,v_2,v_3\in I$ be distinct vertices. Now define $F_i=I\setminus\set{v_i}$ for $i=1,2,3$, and $\pathlet_{jk}=\set{F_j,F_j\cap F_k,F_k}$ for $jk=23,13,12$. For every $F\in\pathlet_{jk}$ we have $v_i\in F$ (where $i\neq j,k$), therefore $F\cup F_i=I\notin\simp K$ and the requirement for asteroidality is fulfilled.\\
If (2) holds, let $(v_0,v_1,v_2,\ldots,v_n=v_0)$ be an induced cycle in $\simp K$ of length $n\geq 4$: the vertices $v_1,\ldots,v_n\in V$ are distinct, $\set{v_{i-1},v_i}\in\simp K$ for $i=1,2,\ldots,n$, and there are no other faces in $\simp K$ among $\set{v_1,\ldots,v_n}$. Let
$$F_1=\set{v_0,v_1}\quad ,\quad F_2=\set{v_1,v_2}\quad ,\quad F_3=\set{v_2,v_3},$$
and let
\begin{equation} \nonumber
\begin{split}
\pathlet_{12}&=\set{\set{v_0,v_1}, \set{v_1}, \set{v_1,v_2}},\\
\pathlet_{23}&=\set{\set{v_1,v_2}, \set{v_2}, \set{v_2,v_3}},\\
\pathlet_{13}&=\set{\set{v_n,v_1}, \set{v_n}, \set{v_{n-1},v_n}, \set{v_{n-1}}, \set{v_{n-2}, v_{n-1}}, \ldots, \set{v_3}, \set{v_2,v_3}}.
\end{split}
\end{equation}
Note that no face $F$ in $\pathlet_{jk}$ is a subset of $F_i$ (where $i\neq j,k$). Therefore $F\cup F_i$ contains $3$ or $4$ vertices, and in particular $\set{v_x,v_y}\su F\cup F_i$ for some non-consecutive $x,y$ (this is where we need $n\geq 4$). Since our cycle is induced, we must have $\set{v_x,v_y}\notin\simp K$, therefore $F\cup F_i\notin\simp K$, proving that the requirement for asteroidality is fulfilled.\\
If (3) holds there is nothing more to prove.
\end{proof}

We can now deduce Theorem \ref{Mat_d=1}:

\begin{proof}[Proof of Theorem \ref{Mat_d=1}]
Every $1$-representable complex is $1$-Matou\v sek (from Proposition \ref{basic_Mat}), so it remains to show that if $\simp K$ is not $1$-representable, it isn't $1$-Matou\v sek either. This immediately follows from Propositions \ref{asteroidal_Mat} and \ref{asteroidal_d=1}.
\end{proof}

\subsection{Representability of graphs and co-graphs}

If $\simp K$ is a complex such that $\dim\simp K'=1$ (i.e. $\simp K'$ is a graph), we informally say that $\simp K$ a co-graph. In these cases, we find that $d$-Matou\v sek and $d$-representability are identical:

\begin{thm} \label{codim_1}
Let $\simp K$ be a simplicial complex, such that $\dim\simp K'=1$. Then for every $d\in\bbN$, $\simp K$ is $d$-representable if and only if it is $d$-Matou\v sek.
\end{thm}

The first case not covered by previously stated results is when $d=2$, and we will show that $2$-representability and $2$-Matou\v sek are both equivalent to the planarity of $\simp K'$. We break this down into a number of lemmas. First, we prove that our configuration space is simply the deleted product of $\simp K'$:

\begin{lem} \label{cograph_dual}
Let $\simp K$ be a simplicial complex such that $\dim\simp K'=1$, i.e. $\simp G:=\simp K'$ is a graph. Then $\hat{\simp K}=\scdp{\simp G}$.
\end{lem}
\begin{proof}
We will prove that for any non-empty faces $\alpha,\beta\in\simp G$ we have
$$\alpha\cap\beta\neq\emp\ \iff\ \Par{\bigcap\alpha}\cup\Par{\bigcap\beta}\in\simp K.$$
This is straightforward: if $J\in\alpha\cap\beta$, then $(\cap\alpha)\su J$ and also $(\cap\beta)\su J$, so $(\cap\alpha)\cup(\cap\beta)\su J\in\simp K$ implying $(\cap\alpha)\cup(\cap\beta)\in\simp K$.\\
On the other hand, suppose $(\cap\alpha)\cup(\cap\beta)\in\simp K$. Then $(\cap\alpha)\cup(\cap\beta)\su J$ for some $\simp K$-maximal face $J$. In particular $\cap\alpha\su J$. Since $\simp G$ is a graph, $\alpha$ is either a singleton or an edge. If $\alpha=\set{J_1}$, then $\cap\alpha\su J$ implies $J_1\su J$. But $J_1\in\m{\simp K}$ is a maximal face in $\simp K$, so $J=J_1$. If $\alpha=\set{J_1,J_2}$ for $J_1\neq J_2$, then $\emp\neq J_1\cap J_2\su J$ (recall that $\alpha\in\simp G=\simp K'$ thus $\cap\alpha\neq\emp$), therefore $\emp\neq J_1\cap J_2=J\cap J_1\cap J_2$ implying $\set{J,J_1,J_2}\in\simp K'=\simp G$. But $\simp G$ is a graph, so in fact $\set{J,J_1,J_2}$ has at most $2$ distinct elements, and $J=J_1$ or $J=J_2$. In every case, we have $J\in\alpha$. The same argument applies for $\beta$, therefore $J\in\beta$ and $\alpha\cap\beta\neq\emp$ as required.
\end{proof}

This allows us to prove our theorem by relying on the following criterion for graph planarity from \cite[``A graph planarity criterion'' p. 597]{Sa}:
\begin{lem}[Sarkaria] \label{graph_del}
Let $\simp G$ be a graph. Then $\simp G$ is planar iff $\ind{\scdp{\simp G}}\leq 1$.
\end{lem}

\begin{proof}[Proof of Theorem \ref{codim_1}]
By Theorem \ref{codim_rep}, every $\simp K$ with $\dim\simp K'=1$ is $3$-representable (and thus also $3$-Matou\v sek by Proposition \ref{basic_Mat}). This clearly implies the same for any $d\geq 3$. For $d=1$ we apply Theorem \ref{Mat_d=1}, so the only remaining case is $d=2$. For this we denote $\simp G=\simp K'$, and prove that the following our equivalent:
\begin{enumerate}
\item $\simp G$ is planar.
\item $\simp K$ is $2$-representable.
\item $\simp K$ is $2$-Matou\v sek.
\end{enumerate}
To prove (1)$\implies$(2), note that any planar graph can by drawn in $\Rt$ using straight line segments, 
so if $\simp G$ is planar we obtain a linear injective map $\geo{\simp G}\to\Rt$, and conclude from Theorem \ref{codim_rep} that $\simp K$ is $2$-representable. The implication (2)$\implies$(3) is simply Proposition \ref{basic_Mat} for $d=2$. To prove (3)$\implies$(1), by Lemma \ref{cograph_dual} we have $\hat{\simp K}=\scdp{\simp G}$. If $\simp K$ is $2$-Matou\v sek then $\ind{\hat{\simp K}}\leq 1$, and Lemma \ref{graph_del} implies that $\simp G$ is planar.
\end{proof}


In contrast to our previous results, for a simplicial complex $\simp K$ which is a a graph, being $2$-representable and being $2$-Matou\v sek are not identical properties. Some previous results regarding dimension are relevant: we know that $\simp K$ is $1$-representable iff it is $1$-Matou\v sek (from Theorem \ref{Mat_d=1}). Every simplicial complex $\simp K$ is $(2\dim\simp K+1)$-representable \cite{TaS}, therefore $\simp K$ is $d$-representable and $d$-Matou\v sek, for every $d\geq 3$. And of course every $2$-representable complex is also $2$-Matou\v sek. But the converse fails: there exist graphs $\simp K$ which are $2$-Matou\v sek but not $2$-representable: see Example \ref{bad_cover}.

There are, however, non-trivial examples of graphs which are not $2$-Matou\v sek. For example, let $\simp K$ be the barycentric subdivision of some non-planar graph $\simp G$. Consider the canonical map $\phi:\geo{\simp G}\to\geo{\simp K'}$ which is the composition of the homeomorphism $\geo{\simp G}\to\geo{\simp K}$ and $\iota_{\simp K}:\geo{\simp K}\to\geo{\simp K'}$ given in Lemma \ref{dual_lemma}. We can verify that applying $\phi$ on each coordinate gives us an equivariant map $\scdp{\simp G}\to\hat{\simp K}$. Since $\simp G$ is not planar, Lemma \ref{graph_del} implies that $\ind{\scdp{\simp G}}=2$, therefore $\simp K$ is not $2$-Matou\v sek by Proposition \ref{basic_index_properties}.


\subsection{$d$-Matou\v sek and $d$-collapsibility}

As we have seen, every $d$-representable complex is both $d$-collapsible and $d$-Matou\v sek (for every $d\in\bbN$). This raises the question of the relation between $d$-collapsibility and $d$-Matou\v sek. We give examples that prove the two properties are independent (neither implies the other), and even together don't coincide with $d$-representable complexes. In other words, there exist simplicial complexes which are $d$-collapsible but not $d$-Matou\v sek, other complexes which are $d$-Matou\v sek but not $d$-collapsible, and others which are both $d$-Matou\v sek and $d$-collapsible but not $d$-representable.

\subsubsection{Complexes that are $d$-collapsible but not $d$-Matou\v sek}
There exist complexes which are $d$-collapsible but not $d$-Matou\v sek even for $d=1$: simply take a complex which is $1$-collapsible but not $1$-representable. Such a complex is not $1$-Matou\v sek by Theorem \ref{Mat_d=1}.

Using asteroidal maps, we may generalize one such complex to arbitrary dimension:
\begin{exmp}
For any $d\in\bbN$, let $\simp L$ have vertices $V=0,1,\ldots,d+2$, and maximal faces given by
$$\m{\simp L}=\set{F\su V\ :\ |F|=d+1\ ,\ 0\in F}.$$
Note that $\simp L$ is the cone of $\Delta_{d+1}^{(d-1)}$. This allows us to define a map $g:\geo{\Delta_{d+1}^{(d)}}\to\geo{\simp L}$: we take $g$ as the identity on any simplex of dimension $\leq d-1$, and extend to $d$-dimensional simplices by contracting the boundary into the vertex $0$.

Now let $\simp K$ be the barycentric subdivision of $\simp L$. It is straightforward to verify that $\simp K$ is $d$-collapsible, and that by composing $g$ with the canonical homeomorphism between $\simp L$ and $\simp K$ we obtain an asteroidal map $f:\geo{\Delta_{d+1}^{(d)}}\to\geo{\simp K}$. Thus by Proposition \ref{asteroidal_Mat}, $\simp K$ is not $d$-Matou\v sek. If $d=1$, we obtain the (unique) minimal non-$1$-representable tree.
\end{exmp}

\comment{
In the previous example(s), we used asteroidal maps to prove that $\simp K$ is not $d$-Matou\v sek. In the following example $\simp K$ is not $2$-Matou\v sek, but a $2$-asteroidal map into $\simp K$ can't exist:

\begin{exmp}
Let $\simp K$ be the simplicial complex with $V=[9]$, and maximal faces:
$$\m{\simp K}\ =\ \set{147\ ,\ 258\ ,\ 369\ ,\ 168\ ,\ 249\ ,\ 357}\ .$$
It is easy to verify that $\simp K$ is $2$-collapsible, and $\simp K'$ is just the complete bipartite graph $K_{3,3}$, therefore it is not $2$-Matou\v sek by Lemmas \ref{cograph_dual}-\ref{graph_del}.
\end{exmp}}

\subsubsection{Complexes that are $d$-Matou\v sek but not $d$-collapsible}

\begin{exmp} \label{Mat_non_col}
In \cite{TaE}, Tancer constructed a good cover in $\Rt$, which has a nerve $\simp K$ that isn't $2$-collapsible. We know that $\simp K$ is $2$-Matou\v sek from Corollary \ref{good_cover_Mat}.
\end{exmp}

\subsubsection{Complexes that are $d$-Matou\v sek and $d$-collapsible, but not $d$-representable}

The previous examples show that $d$-Matou\v sek and $d$-collapsibility are ``independent'' properties of $d$-representable complexes. Even combining both notions, we still remain with some complexes which are not $d$-representable, even when	 $d=2$. We give two examples of such complexes, one which is the nerve of a good cover in $\Rt$, and one which is not.

\begin{exmp}
Let $\simp K$ be the triangulation of the Mobius strip, with vertices $V=[6]$, and maximal faces
$$\m{\simp K}\ =\ \set{124\ ,\ 125\ ,\ 156\ ,\ 235\ ,\ 236\ ,\ 246\ ,\ 345\ ,\ 456}.$$
This complex was given by Eckhoff in \cite{Eck} as an example of a (strongly) $2$-collapsible complex that is not $2$-representable. The complex is the nerve of a good cover in $\Rt$, outlined in Figure \ref{mobious_watch}:

\begin{center}
\refstepcounter{figure_num}\label{mobious_watch}
\scriptsize
\begin{tikzpicture}[scale=1.25]
\coordinate (x235) at (0,0);
\coordinate (x125) at (0,2);
\coordinate (x156) at (1.902,2.618);
\coordinate (x124) at (-1.902,2.618);
\coordinate (x345) at (1.902,-0.618);
\coordinate (x236) at (-1.902,-0.618);
\coordinate (x456) at (3.078,1);
\coordinate (x246) at (-3.078,1);

\coordinate (out1) at (4,1.5);
\coordinate (mid1) at (4,1);
\coordinate (inn1) at (4,0.5);
\coordinate (out2) at (6,-0.5);
\coordinate (mid2) at (5.5,-0.5);
\coordinate (inn2) at (5,-0.5);
\coordinate (out3) at (4,-2.5);
\coordinate (mid3) at (4,-2);
\coordinate (inn3) at (4,-1.5);
\coordinate (out4) at (2,-2.5);
\coordinate (mid4) at (2,-2);
\coordinate (inn4) at (2,-1.5);
\coordinate (out5) at (-2,-2.5);
\coordinate (mid5) at (-2,-2);
\coordinate (inn5) at (-2,-1.5);
\coordinate (out6) at (-4,-2.5);
\coordinate (mid6) at (-4,-2);
\coordinate (inn6) at (-4,-1.5);
\coordinate (out7) at (-6,-0.5);
\coordinate (mid7) at (-5.5,-0.5);
\coordinate (inn7) at (-5,-0.5);
\coordinate (out8) at (-4,1.5);
\coordinate (mid8) at (-4,1);
\coordinate (inn8) at (-4,0.5);

\node (x1) at (0,2.35) {\large {\textbf {1}}};
\node (x2) at (-1.376,1) {\large {\textbf {2}}};
\node (x3) at (0,-0.35) {\large {\textbf {3}}};
\node (r4) at (3.3, 0.6) {\large {\textbf {4}}};
\node (l4) at (-3.3, 1.4) {\large {\textbf {4}}};
\node (br4) at (2, -1.75) {\large {\textbf {4}}};
\node (bl4) at (-2, -2.25) {\large {\textbf {4}}};
\node (x5) at (1.376,1) {\large {\textbf {5}}};
\node (r6) at (3.3, 1.4) {\large {\textbf {6}}};
\node (l6) at (-3.3, 0.6) {\large {\textbf {6}}};
\node (br6) at (2, -2.25) {\large {\textbf {6}}};
\node (bl6) at (-2, -1.75) {\large {\textbf {6}}};

\draw (x235) -- (x125);
\draw (x125) -- (x156);
\draw (x156) -- (x456);
\draw (x456) -- (x345);
\draw (x345) -- (x235);
\draw (x125) -- (x124);
\draw (x124) -- (x246);
\draw (x246) -- (x236);
\draw (x236) -- (x235);
\draw (x124) -- (x156);
\draw (x236) -- (x345);
\draw (x156) to [out=0, in=180] (out1) to [out=0, in=90] (out2) to [out=270, in=0] (out3) to [out=180, in=0] (out4) to [out=180, in=0] (inn5) to [out=180, in=0] (inn6) to [out=180, in=270] (inn7) to [out=90, in=180] (inn8) to [out=0, in=180] (x236);
\draw (x456) to (mid1) to [out=0, in=90] (mid2) to [out=270, in=0] (mid3) to (mid6) to [out=180, in=270] (mid7) to [out=90, in=180] (mid8) to (x246);
\draw (x345) to [out=0, in=180] (inn1) to [out=0, in=90] (inn2) to [out=270, in=0] (inn3) to [out=180, in=0] (inn4) to [out=180, in=0] (out5) to [out=180, in=0] (out6) to [out=180, in=270] (out7) to [out=90, in=180] (out8) to [out=0, in=180] (x124);

\path[fill, opacity=0.1] (x124) to (x246) to (x236) to (x345) to (x456) to (x156) to cycle;
\path[fill, opacity=0.1] (x156) to [out=0, in=180] (out1) to [out=0, in=90] (out2) to [out=270, in=0] (out3) to [out=180, in=0] (out4) to [out=180, in=0] (inn5) to [out=180, in=0] (inn6) to [out=180, in=270] (inn7) to [out=90, in=180] (inn8) to [out=0, in=180] (x236) to (x246) to (x124) to [out=180, in=0] (out8) to [out=180, in=90] (out7) to [out=270, in=180] (out6) to [out=0, in=180] (out5) to [out=0, in=180] (inn4) to [out=0, in=180] (inn3) to [out=0, in=270] (inn2) to [out=90, in=0] (inn1) to [out=180, in=0] (x345) to (x456) to cycle;

\end{tikzpicture}
\\Figure \arabic{figure_num}\\
\end{center}

The sets $U_1,\ldots,U_6\su\Rt$ should be open balls that slightly overlap the boundaries indicated in the figure. We deduce that $\simp K$ is $2$-Matou\v sek from Corollary \ref{good_cover_Mat}.
\end{exmp}

\begin{exmp} \label{bad_cover}
We will now give an example of a graph $\simp K$ which is $2$-collapsible and $2$-Matou\v sek, but not the nerve of any good cover in $\Rt$.\\
A graph $\simp G$ is a \emph{string graph} if it is the nerve of a collection of curves in $\Rt$. In \cite{KM}, Kratochv\'il and Matou\v sek constructed string graphs for which the number of intersection points between any such collection of curves in $\Rt$ is exponential in the number of curves. One instance of this construction will be our example.
Consider the graph $\simp G$ drawn in $\Rt$ in \cite{KM} (some vertices have been renamed), reproduced as Figure \ref{string_example} below:\pagebreak
\begin{center}
\refstepcounter{figure_num}\label{string_example}
\begin{tikzpicture}[scale=1.8]
\coordinate(c0) at (0,0);
\coordinate(c1) at (1,0);
\coordinate(c2) at (2,0);
\coordinate(c3) at (4,0);
\coordinate(c4) at (7,0);
\coordinate(l) at (0,1);
\coordinate(u2) at (1,1);
\coordinate(u1) at (2,1);
\coordinate(v2) at (3,1);
\coordinate(u0) at (4,1);
\coordinate(v1) at (5.5,1);
\coordinate(r) at (7,1);
\coordinate(d0) at (0,4);
\coordinate(a) at (4,4);
\coordinate(b) at (4,2);
\coordinate(d1) at (7,4);

\draw (c0) -- (c4);
\draw (l) -- (r);
\draw (d0) -- (d1);
\draw (c0) -- (d0);
\draw (c1) -- (u2);
\draw (c2) -- (u1);
\draw (c3) -- (a);
\draw (c4) -- (d1);

\coordinate (ab_center) at (4,3);
\coordinate (ab_low) at (4,2.5);
\coordinate (ab_top) at (4,3.5);
\coordinate (v1y) at (6.25,1);
\coordinate (v1_low) at (5.5,0.3);
\coordinate (u0v1) at (4.75,1);

\draw[dashed] (u1) to [out=90, in=180] (ab_center) to [out=0, in=90] (v1);
\draw[dashed]
(u2) to [out=90, in=180] (ab_top) to [out=0, in=90] (v1y) to [out=270, in=0] (v1_low) to [out=180, in=270] (u0v1) to [out=90, in=0] (ab_low) to [out=180, in=90] (v2);

\draw[fill=white] (c0) circle (1.25pt) node[below left] {$c_0$};
\filldraw (c0) circle (1.25pt) node[below left] {$c_0$};
\filldraw (c1) circle (1.25pt) node[below] {$c_1$};
\filldraw (c2) circle (1.25pt) node[below] {$c_2$};
\filldraw (c3) circle (1.25pt) node[below] {$c_3$};
\filldraw (c4) circle (1.25pt) node[below right] {$c_4$};
\filldraw (l) circle (1.25pt) node[left] {$l$};
\filldraw (u2) circle (1.25pt) node[above left] {$u_2$};
\filldraw (u1) circle (1.25pt) node[above left] {$u_1$};
\filldraw (v2) circle (1.25pt) node[above left] {$v_2$};
\filldraw (u0) circle (1.25pt) node[above right] {$u_0$};
\filldraw (v1) circle (1.25pt) node[below] {$v_1$};
\filldraw (r) circle (1.25pt) node[right] {$r$};
\filldraw (d0) circle (1.25pt) node[above left] {$d_0$};
\filldraw (a) circle (1.25pt) node[above] {$a$};
\filldraw (b) circle (1.25pt) node[left] {$b$};
\filldraw (d1) circle (1.25pt) node[above right] {$d_1$};

\end{tikzpicture}
\scriptsize\\
Figure \arabic{figure_num}
\end{center}
Note that $\simp G$ includes the edges drawn using dashed lines ($u_1v_1$ and $u_2v_2$) as well as those drawn with solid lines.

We define our simplicial complex $\simp K$ as the barycentric subdivision of $\simp G$, with the addition of pairs of $\simp G$-edges which intersect in Figure \arabic{figure_num}. In other words $V(\simp K)=V(\simp G)\cup E(\simp G)$, and the maximal faces of $\simp K$ are all pairs of the form $\set{x,xy}$ where $xy$ is an edge in $\simp G$, along with $4$ additional pairs:
$$\ \ \set{ab,u_1v_1}\quad,\quad\set{ab,u_2v_2}\quad,\quad\set{u_0v_1,u_2v_2}\quad,\quad\set{v_1r,u_2v_2}\ .$$
Note that $\simp K$ is itself a graph (maximal faces of $\simp K$ are $1$-dimensional), so $\simp K$ is trivially $2$-collapsible. 
Our next goal is to prove that $\simp K$ is $2$-Matou\v sek (in fact we show that there exists a faithful map $\geo{\simp K'}\to\Rt$), and that $\simp K$ is not the nerve of any good cover in $\Rt$. Both of these facts can be deduced by examining the drawing of $\simp G$ in $\Rt$.

To ease notation, for every $x\in V(\simp G)$, let $\overline x\in\Rt$ denote the position of $x$ indicated by the drawing, and for every edge $xy\in E(\simp G)$ let $\overline{xy}\su\Rt$ denote all points on the indicated path from $\overline x$ to $\overline y$.

The drawing shows that $\simp K$ is ``almost'' the nerve of a good cover in $\Rt$: for every $x\in V(\simp G)$, let $U_x\su\Rt$ be an open ball of radius $\varepsilon>0$ centered around $\overline x$. For every $e=xy\in E(\simp G)$, let $U_e\su\Rt$ be the $\frac{\varepsilon}{10}$-neighborhood of the points in $\overline{xy}\setminus\Par{U_x\cup U_y}$. We may choose $\varepsilon$ to be sufficiently small as not to introduce unwanted intersections (i.e. intersections among collections not in $\simp K$). For example, $U_{ab}\cap U_b$ and $U_{bu_0}\cap U_b$ must be nonempty, but $U_{ab}\cap U_{bu_0}=\emp$. Clearly the \emph{only} intersection of sets among the collection $\set{U_w}_{w\in V(\simp K)}$ which is neither empty nor contractible is $U_{ab}\cap U_{u_2v_2}$, which is a disjoint union of $2$ contractible sets. Nevertheless, the proof of Lemma \ref{good_cover_faithful} still applies (see Remark \ref{almost_good_cover}), we obtain a faithful map $\geo{\simp K'}\to\Rt$, and $\simp K$ is $2$-Matou\v sek by Proposition \ref{faithful_Mat}.

Finally, assume by contradiction that $\simp K$ is the nerve of a collection $\set{U_w}_{w\in V(\simp K)}$ which is a good cover in $\Rt$.
We draw $\simp G$ in $\Rt$, (that is we construct a map $\geo{\simp G}\to\Rt$), such that every vertex $x\in V(\simp G)$ is mapped to some $\overline x\in U_x$, and every edge $e=xy\in E(\simp G)$ is drawn as a piecewise linear curve $\overline{xy}\su U_x\cup U_e\cup U_y$.

Let $\simp H$ be the graph obtained from $\simp G$ by removing the edges $u_1v_1$ and $u_2v_2$ (so $\simp H$ is the graph drawn in Figure \arabic{figure_num} without dashed edges).
Note that the induced drawing of $\simp H$ must be planar, and $\simp H$ has the same faces in every planar drawing (so the faces are the same as in Figure \arabic{figure_num}). Specifically, $\overline{u_1v_1}$ must cross $\overline{ab}$, and $U_{ab}\setminus U_{u_1v_1}$ must have (at least) two connected components (one bounded in the interior of the curve $\overline{u_1}\overline{v_1}\overline{u_0}\overline{v_2}\overline{u_1}$, the other in the exterior). Finally, observe that $\overline{u_2v_2}$ must intersect $\overline{ab}$ at least twice (once in each connected component) as in Figure \arabic{figure_num} - thus $U_{ab}\cap U_{u_2v_2}$ is not path-connected, contradicting our assumption.
\end{exmp}


\section{Further questions}  

\begin{ques}
Are there $d$-Matou\v sek complexes that aren't $d$-Leray?
\end{ques}
We know of the complex in Example \ref{Mat_non_col} due to Tancer, which is $2$-Matou\v sek but not $2$-collapsible. But it is unknown to the author whether there exists a simplicial complex which is $d$-Matou\v sek but not $d$-Leray, even for $d=2$. By Remark \ref{induced_sub}, the question is equivalent to: does there exist a $d$-Matou\v sek complex with non-trivial $m$-homology for some $m\geq d$?

Furthermore, if such complexes exist, can we find a property similar to $d$-Matou\v sek, which also implies that a complex is $d$-Leray?

\begin{ques}
Under what conditions is a $d$-Matou\v sek complex also $d$-representable?
\end{ques}
For example, the Haefliger-Weber theorem (see \cite[Theorem 5.5]{Sk}) states that if $2d\geq 3\Par{\dim\simp L+1}$ and there exists an equivariant map $\scdp{\simp L}\to S^{d-1}$, then there also exists an embedding $\geo{\simp L}\to\Rd$.
This implies that if $\simp K$ satisfies $\hat{\simp K}=\scdp{\Par{\simp K'}}$ (such as those $\simp K$ given by Lemma \ref{dual_classifying}) and $2d\geq 3\Par{\dim\simp K'+1}$, then $\simp K$ is $d$-Matou\v sek iff $\simp K$ is the nerve of a good cover in $\Rd$.

In general, the configuration space $\hat{\simp K}$ can be ``smaller'' than $\scdp{\Par{\simp K'}}$. Trying to modify Haefliger-Weber for our needs, we may apply Skopenkov's proof of \cite[Proposition 8.5]{Sk} to $\hat{\simp K}$ and find that for general $\simp K$ and $2d\geq 3\Par{\dim\simp K'+1}$, the complex $\simp K$ is $d$-Matou\v sek if and only if there exists a faithful map $f:\geo{\simp K'}\to\Rd$. However, non-empty intersections of $\simp K'$-faces introduced by such $f$ obtained from the proof need not be contractible.



\begin{ques}
What new results can we obtain by using more general methods?
\end{ques}

While stronger topological methods exist (see \cite{Zi}), we choose to use the Borsuk--Ulam method as it doesn't require as much intensive topology (the Borsuk--Ulam theorem is the only topological result needed), succeeds in proving new results and outlining limitations of topological methods, but possibly not reaching their full potential. We therefore welcome improvements to the ``correct'' definition of a $d$-Matou\v sek complex. In vague terms, if a $d$-Leray complex is one which has no homological reason not to be $d$-representable, a $d$-Matou\v sek complex as defined in this paper is one for which a $d$-representation lifted to the configuration space of pairs doesn't contradict the Borsuk--Ulam theorem. The true spirit of the definition for a $d$-Matou\v sek complex should be a complex which doesn't exhibit similarly defined topological barriers to $d$-representability.

\bibliographystyle{plain}

\begin{thebibliography}{1}


\bibitem{BB}
E. G. Bajm\'oczy, I. B\'ar\'any, \textit{On a common generalization of Borsuk's and Radon's theorem}, Acta Mathematica Hungarica, \textbf{34} (1979), 347--350.

\bibitem{Eck}
J. Eckhoff, \textit{An upper-bound theorem for families of convex sets}, Geometriae Dedicata, \textbf{19} (1985), 217--227.

\bibitem{FWZ}
M. Filakovsk\'y, U. Wagner, S. Zhechev, \emph{Embeddability of simplicial complexes is undecidable}, In: Proceedings of the Fourteenth Annual ACM-SIAM Symposium on Discrete Algorithms, Society for Industrial and Applied Mathematics, 2020, pp 767--785.


\bibitem{KM}
J. Kratochv\'il, J. Matou\v sek, \textit{String graphs requiring exponential representations}, Journal of Combinatorial Theory, Series B, \textbf{53} (1991), 1--4.

\bibitem{LB}
C. Lekkerkerker, J. Boland, \textit{Representation of a finite graph by a set of intervals on the real line}, Fundamenta Mathematicae, \textbf{51} (1962), 45--64.

\bibitem{MaH}
J. Matou\v sek, \textit{A Helly-type theorem for unions of convex sets}, Discrete \& Computational Geometry, \textbf{18} (1997), 1--12.

\bibitem{Ma}
J. Matou\v sek, \textit{Using the Borsuk--Ulam theorem: lectures on topological methods in combinatorics and geometry}, Springer, Berlin, 2008.

\bibitem{Sa}
K. S. Sarkaria, \textit{Kneser colorings of polyhedra}, Illinois Journal of Mathematics, \textbf{33} (1989), 592--620.

\bibitem{Sk}
A. B. Skopenkov, \textit{Embedding and knotting of manifolds in Euclidean spaces}, In: Surveys in contemporary mathematics, London Mathematical Society Lecture Note Series 347, Cambridge University Press, Cambridge, 2008, pp. 248--342.

\bibitem{TaE}
M. Tancer, \textit{A counterexample to Wegner’s conjecture on good covers}, Discrete \& Computational Geometry, \textbf{47} (2012), 266--274.

\bibitem{TaRep}
M. Tancer, \textit{d-representability of simplicial complexes of fixed dimension}, Journal of Computational Geometry, \textbf{2} (2011), 183--188.

\bibitem{TaS}
M. Tancer, \textit{Intersection patterns of convex sets via simplicial complexes: a survey}, In: Thirty essays on geometric graph theory, Springer, New York, 2013, pp. 521--540.

\bibitem{TT}
M. Tancer, D. Tonkonog, \emph{Nerves of good covers are algorithmically unrecognizable}, SIAM Journal on Computing, \textbf{42} (2013), 1697--1719.

\bibitem{We}
G. Wegner, \textit{d-Collapsing and nerves of families of convex sets}, Archiv der Mathematik, \textbf{26} (1975), 317--321.

\bibitem{Zi}
R.T. \v Zivaljevi\'c, \textit{User's guide to equivariant methods in combinatorics. II.}, Publications de l'Institut Math\'ematique. Nouvelle S\'erie, \textbf{64} (1998), 107--132.

\end{thebibliography}

\end{document}